%% file: main.tex
\documentclass[12pt]{colt2022}
\pdfoutput=1

\input{prelude.tex}

\title[Robust and Private Estimation of Discrete Distributions]{Robust Estimation of Discrete Distributions under Local Differential Privacy}

\coltauthor{%
 \Name{Julien Chhor\thanks{Equal contributions} } \Email{julien.chhor@ensae.fr}\\
 \addr CREST/ENSAE
 \AND
 \Name{Flore Sentenac$^*$} \Email{flore.sentenac@ensae.fr}\\
 \addr CREST/ENSAE
}

\usepackage{times}

\begin{document}

\maketitle

\begin{abstract}%
Although robust learning and local differential privacy are both widely studied fields of research, combining the two settings is an almost unexplored topic. We consider the problem of estimating a discrete distribution in total variation from $n$ contaminated data batches under a local differential privacy constraint.  
A fraction $1-\epsilon$ of the batches contain $k$ i.i.d. samples drawn from a discrete distribution $p$ over $d$ elements. To protect the users' privacy, each of the samples is privatized using an $\alpha$-locally differentially private mechanism. The remaining $\epsilon n $ batches are an adversarial contamination. The minimax rate of estimation under contamination alone, with no privacy, is known to be $\epsilon/\sqrt{k}+\sqrt{d/kn}$, up to a $\sqrt{\log(1/\epsilon)}$ factor. Under the privacy constraint alone, the minimax rate of estimation is $\sqrt{d^2/\alpha^2 kn}$. We show that combining the two constraints leads to a minimax estimation rate of $\epsilon\sqrt{d/\alpha^2 k}+\sqrt{d^2/\alpha^2 kn}$ up to a $\sqrt{\log(1/\epsilon)}$ factor, larger than the sum of the two separate rates. 
We provide a polynomial-time algorithm achieving this bound, as well as a matching information theoretic lower bound. 
\end{abstract}

\hfill


\input{contents/Introduction}

\input{contents/Setting}

\input{contents/Results}

\section{Upper bound}

\input{contents/UpperBoundSketch}


\input{contents/Discussion_and_future_work}

\bibliography{contents/biblio}

\section{Appendix}
\input{contents/JainOrlinskyAppendix}

\input{contents/proof_cor_actual_estimator}
\input{contents/LB_Robustness_Privacy}

\input{contents/LB_privacy_no_outliers}

\end{document}

%% file: prelude.tex
\usepackage[utf8]{inputenc}

\usepackage{amsfonts}
\usepackage{dsfont}

\usepackage{framed}

\allowdisplaybreaks

\usepackage[english]{babel}
\usepackage{framed}
\usepackage{tikz}
\usepackage{bbold}

\usepackage{mathtools}
\usepackage{commath}
\usepackage{graphicx} 
\DeclareMathOperator*{\argmax}{arg\,max}

\newcommand{\robustsubroutine}{ \textsc{Robust Estimation Procedure}}

\newcommand{\R}{\mathbb{R}}
\newcommand{\N}{\mathbb{N}}

\newcommand{\E}{\mathbb{E}}
\newcommand{\V}{\mathbb{V}}

\newcommand{\C}{\mathbf{C}}
\newcommand{\G}{\mathcal{G}}

\newcommand{\hatC}{\widehat{\C}}
\newcommand{\cov}{\text{Cov}}
\newcommand{\hatcov}{\widehat{\text{Cov}}}
\newcommand{\var}{\mathbf{V}}
\newcommand{\hatvar}{\widehat{\var}}
\newcommand{\indS}[1]{\mathds{1}_{#1}}
\newcommand{\hatq}{\widehat{q}}

\newcommand{\mub}{\widehat{q}}
\newcommand{\cgamma}{c_\gamma}

\newcommand{\rappor}{\textsc{RAPPOR} }

\newcommand{\Indist}{\mathcal{A}}
\newcommand{\Indistkhi}{\Indist_{\chi^2}}

\newcommand{\CL}[1]{C_{#1}}

\usepackage{thmtools} 
\usepackage{thm-restate}

\usepackage[symbol]{footmisc}

\renewcommand{\thefootnote}{\fnsymbol{footnote}}

%% file: contents/Introduction.tex
\section{Introduction}

In recent machine learning developments, the growing need to analyze potentially corrupted, biased or sensitive data has given rise to unprecedented challenges.
To extract relevant information from today's data, studying algorithms under new learning constraints has emerged as a major necessity. 
To name a few, let's mention learning from incomplete data, transfer learning, fairness, robust learning or privacy. 
Although each one of them has been subject to intense progress in recent works, combining several learning constraints is not conventional. In this work, we propose to study how to estimate discrete distributions under the constraint of both being robust to adversarial contamination and of ensuring local differential privacy.\\

On the one hand, robust learning has received considerable attention over the past decades. 
This recent research has been developing in two main directions.
The first one deals with robustness to heavy tails, see~\cite{catoni2012challenging}, see also \cite{lugosi2019mean} for an excellent review. The second one explores robustness to outliers. 
It mainly considers two contamination models, which are the Huber contamination model~\cite{huber1968robust}, \cite{huber1992robust}, \cite{huber2004robust}, \cite{huber2009robust}, where the outliers are iid with an unknown probability distribution, and the \textit{adversarial} contamination, where the outliers are added by a malicious adversary who knows the estimation procedure, the underlying distribution and the data and seeks to deteriorate the procedure's estimation performance (\cite{diakonikolas2019robust,rousseeuw2011robust,dalalyan2020all,lecue2020robust}). \\

On the other hand, preserving the privacy of individuals has emerged as a major concern, as more and more sensitive data are collected and processed. The most commonly used privatization framework is that of differential privacy (\cite{dwork2006calibrating}, \cite{butucea2020local}, \cite{lam2020minimax}, \cite{berrett2020locally}, \cite{cai2019cost}). Both central and local models of privacy are considered in the field. 
In the centralized case, a global entity collects the data and analyzes it before releasing a privatized result, from which the original data should not be possible to infer. In local privacy, the data themselves are released and should remain private (\cite{duchi2014local}). The paper focuses on the latter notion. A vast line of work also studies private mechanism under communication constraints (\cite{https://doi.org/10.48550/arxiv.2007.10976}, \cite{https://doi.org/10.48550/arxiv.2010.06562}, \cite{https://doi.org/10.48550/arxiv.2011.00083}), which we do not consider here, but adding a communication constraint would be interesting future work. \\

Connections between robustness and \textit{global} differential privacy have been recently discussed in
(\cite{pinot2019unified,lecuyer2019certified}, \cite{naseri2020toward}). These papers show that the two notions rely on the same theoretical concepts, and that results in the two fields are related. In other words, robustness and \textit{global} differential privacy work well together. Several papers developed algorithms under robustness and \textit{global} differential privacy constraints (\cite{liu2021robust}, \cite{https://doi.org/10.48550/arxiv.2111.12981}, \cite{https://doi.org/10.48550/arxiv.2111.06578}, \cite{https://doi.org/10.48550/arxiv.2111.11320}).\\

In this paper, we study how \textit{local} differential privacy interacts with robustness. This interaction has been studied previously in \cite{https://doi.org/10.48550/arxiv.1909.09630}, where the authors provide upper and lower bound for estimating discrete distributions under the two constraints. The lower bound was later tightened in \cite{pmlr-v139-acharya21a}. The papers also study testing. We detail in Section \ref{sec:relatedwork} how our setting is a generalisation of theirs. The work of \cite{li2022robustness} also considers this interaction. We explain in more details the differences between their setting and ours in Section \ref{sec:relatedwork}.\\

In this paper, we study how to combine robust statistics with local differential privacy for estimating discrete distributions over finite domains. 
Assume that we want to gather information from $n$ data centers (think of $n$ hospitals for instance). For each of them, we collect $k$ iid observations with unknown discrete distribution $p$ to be estimated. 
To protect the users' privacy (patients data in the hospital example), each single one of the $nk$ observations is privatized using an $\alpha$-locally differentially private mechanism (see the formal definition of local differential privacy in Subsection \ref{subsec:definitions}). 
However, an $\epsilon$-fraction of the data centers are untrustworthy and can send adversarially chosen data.
The goal is to estimate $p$ in total variation distance (or $\ell_1$ distance) from these $n$ corrupted and privatized batches of size $k$. 
This setting is quite natural, as in many applications, the data are collected in batches, some of which may be untrustworthy or even adversarial.

\subsection{Related work}\label{sec:relatedwork}

With the local differential privacy constraint only (i.e. without contamination), the problem of estimating discrete distributions has been solved in (\cite{10.1145/773153.773174}, \cite{https://doi.org/10.48550/arxiv.0803.0924})where the authors propose a polynomial-time and minimax optimal algorithm for estimation under $\ell_1$ and $\ell_2$ losses. 
Note that \textit{without privacy and outliers}, the minimax estimation rate in $\ell_1$ is known to be $\sqrt{\frac{d}{N}}$, where $N$ is the number of iid samples with a discrete distribution over $d$ elements (see, e.g.~\cite{han2015minimax}). The paper \cite{duchi2014local} shows that under privacy alone, the $\ell_1$ minimax rate scales as~$\frac{d}{\alpha\sqrt{N}}$. We give an alternative proof of the lower bound of~\cite{duchi2014local}, in Appendix \ref{app:proofLBprivacyNoOutliers}.\\

With the robustness constraint only (i.e. with $n$ adversarially corrupted batches but without local differential privacy), the problem of estimating discrete distributions has been considered in~\cite{qiao2017learning}. 
For $k = 1$, it is well known that $\Omega(\epsilon)$ error is unavoidable. However,~\cite{qiao2017learning} surprisingly prove that the error can be reduced provided that $k$ is large enough.
More precisely, they show that with no privacy but under contamination, the minimax risk of estimation under $\ell_1$ loss from $n$ batches of size $k$ and $\epsilon$ adversarial corruption on the batches scales as $\sqrt{\frac{d}{N}} + \frac{\epsilon}{\sqrt{k}}$, where $N = nk$.
\cite{qiao2017learning} both provide an information theoretic lower bound and a minimax optimal algorithm, unfortunately running in exponential time in either $k$ or $d$. 
Polynomial-time algorithms were later proposed by \cite{chen2020efficiently}, \cite{jain2020optimal} and were shown to reach the information theoretic lower bound up to an extra $\sqrt{\log(\frac{1}{\epsilon})}$ factor.
In this specific setting, it is not known if this extra factor represents a computational gap between polynomial-time and exponential-time algorithms. 
However, for the problem of robust mean estimation of \textit{normal distributions}, some lower bounds suggest that this exact quantity cannot be removed from the rate
of computationally tractable estimators (see~\cite{diakonikolas2017statistical}).\\

Closer to our setting, the papers by~\cite{https://doi.org/10.48550/arxiv.1909.09630}, \cite{pmlr-v139-acharya21a} and \cite{li2022robustness} combine robustness with local differential privacy. The problem studied here is a generalisation of the first two papers where the authors consider un-batched data, which corresponds to $k=1$ in our setting. 
The setting considered by~\cite{li2022robustness} is not the same as ours, as do not consider discrete distributions and implicitly assume $k=1$.
More importantly, in their setting, contamination comes \textit{before} privacy: some of the raw data $X_1,\dots, X_n$ are outliers themselves, and the privacy mechanism is applied on each $X_i$. Conversely, in our work and in the previous two papers, contamination occurs \textit{after} privacy: none of the raw data are outliers and the adversary is allowed to choose the contamination directly on the set of privatized data.
As we will highlight below, this difference yields fundamentally different phenomena compared to the results in~\cite{li2022robustness}.

\subsection{Summary of the contributions}

In this paper, we study the interplay between local differential privacy and \textit{adversarial} contamination, when the contamination comes \textit{after} the data have been privatized. In this case, we prove that the resulting estimation rate is not merely the sum of the two estimation rates stated in~\cite{duchi2014local} and~\cite{qiao2017learning} but is always slower. More specifically, the term due to the contamination in the bound suffers a multiplicative inflation of $\sqrt{d}/\alpha$. This generalizes a phenomenon first observed in ~\cite{https://doi.org/10.48550/arxiv.1909.09630}.  This phenomenon stands in contrast with~\cite{li2022robustness}, for which the resulting rate is exactly the sum of the rate with privacy but no contamination, plus the rate with contamination but no privacy. The reason is that in~\cite{li2022robustness}, contamination occurs \textit{before} privacy. We provide an explicit algorithm that returns an estimator achieving the optimal bound up to a factor $\sqrt{\log(1/\epsilon)}$, and runs polynomially in all parameters. This algorithm is an adaptation to our setting of methods that were previously used for robust estimation of discrete distributions (\cite{jain2020optimal, jain2021robust}). On a side note, the algorithms introduced in \cite{https://doi.org/10.48550/arxiv.1909.09630} and \cite{acharya2021robust} require the use of a public coin. The proposed algorithm also holds in their setting and relieves this assumption.

%% file: contents/Setting.tex
\section{Setting}

\subsection{Definitions}\label{subsec:definitions}

For any integer $d \geq 2$, denote by $\mathcal{P}_d = \Big\{p \in \R^d ~\big| ~  \forall j : p_j \geq 0 ~\text{ and } \sum_{j=1}^d p_j = 1\Big\}$ the set of probability vectors over $\{1,\dots,d\}$. 
For any $x \in \R^d$, we write $\|x\|_1 = \sum\limits_{j \in [d]} |x_j|$ and $\|x\|_2^2 = \sum\limits_{j \in [d]} x_j^2$. For any two probability distributions $p,q$ over some measurable space $(\mathcal{X}, \mathcal{A})$, we denote by
\begin{align*}
    TV(p,q) &= \sup\limits_{A \in \mathcal{A}} |p(A) - q(A)| \text{  the total variation between $p$ and $q$.}
\end{align*}

Fix $\alpha \in (0,1)$ and consider two measurable spaces $(\mathcal{X}, \mathcal{A})$ and $(\mathcal{Z},\mathcal{B})$. A Markov transition kernel $Q : (\mathcal{X}, \mathcal{A}) \to (\mathcal{Z},\mathcal{B})$ is said to be a (non-interactive) $\alpha$-\textit{locally differentially private} mechanism if it satisfies
\begin{equation}\label{def_LDP}
    \sup_{B \in \mathcal{B}} ~ \sup_{x,x' \in \mathcal{X}} ~ \frac{Q(B|x)}{Q(B|x')} \leq e^\alpha.
\end{equation}

For any $x \in \mathcal{X}$, we say that the random variable $Z$ is a privatized version of $x$ if $Z \sim Q(\cdot|x)$. The measurable space $(\mathcal{Z},\mathcal{B})$ is called the \textit{image space} of $Q$. 
In what follows, we use the Landau notation $O$ which hides an absolute constant, independent of $d, \epsilon, n ,k, \alpha, Q, p$.

\subsection{Model}

We consider the problem of learning a discrete distribution $p$ over a finite set $\{1,\dots,d\}$, $d \geq 3$ under two learning constraints: a) ensuring $\alpha$-local differential privacy and b) being robust to adversarial contamination. To this end, we assume that the data are generated as follows. For some small enough absolute constant $c \in (0,\frac{1}{4})$ and for some known corruption level $\epsilon \in (0,c)$, we will use the notation $n' = n(1-\epsilon)$ throughout and assume that $n' \in \N$.

\begin{enumerate}
    \item First, $n'$ iid \textit{batches} of observations $X^1,\dots,X^{n'}$ are collected.
    More precisely, each batch $X^b$ can be written as $X^b = (X_1^b, \dots, X_k^b)$ and consists of $k$ iid random observations with an unknown discrete distribution $p \in \mathcal{P}_d$, i.e. $\forall (b,l,j) \in [n'] \times [k] \times [d]:  \mathbb{P}(X_l^b = j) = p_j$.
    \item Second, we privatize each of the $n'k$ observations using an $\alpha$-LDP mechanism $Q$, yielding $n'$ iid batches $Y^1, \dots, Y^{n'}$ such that $Y^b = (Y_1^b, \dots, Y_k^b)$ where $Y_l^b| X_l^b \sim Q(\cdot|X_l^b)$. 
    We denote by $Qp$ the distribution of any random variable $Y_l^b$. 
    We then have: $Qp(dz) = \sum\limits_{j\in[d]} p_j Q(dz|j)$, where $Q(dz|j)$ is a shorthand for $Q(dz|X=j)$. 
    The mechanism $Q$ is chosen by the statistician in order to preserve statistical performance while ensuring privacy.
    \item An adversary is allowed to build $n\epsilon$ batches $Y^{n'+1}, \dots, Y^n$ on which no restriction is imposed. Then, he shuffles the set of $n$ batches $(Y_1,\dots,Y_n)$. The resulting set of observations, denoted as $B = (Z^1, \dots, Z^n)$, is referred to as the $\epsilon$-corrupted family of batches.
\end{enumerate}

The observed dataset therefore consists of $n = |B|$ batches of $k$ samples each. Among these batches is an unknown collection of \textit{good batches}
$B_G \subset B$ of size $n(1-\epsilon)$, corresponding to the non-contaminated batches. The remaining set $B_A = B\setminus B_G$ of size $n\epsilon$, denotes the unknown set of adversarial batches. \\

The statistician never has access to the actual observations $X^1,\dots, X^{n'}$, but only to $Z^1, \dots, Z^n$ where $Z^b = (Z_1^b, \dots, Z_k^b)$. 
Each batch is assumed to be either entirely clean or adversarially corrupted. 
Note that observing $n$ batches of size $k$ encompasses the classical case where $k=1$, for which the data consist of $n$ iid and $\epsilon$-corrupted \textit{single observations} rather than batches. 
On top of being more general, the setting with general $k$ allows us to derive faster rates for large $k$ than for the classical case $k=1$.
Note also that in our setting, the contamination comes after the data have been privatized, which is one of the main differences with~\cite{li2022robustness}, where the authors assume that the Huber contamination comes before privacy. 
The examples considered by the authors are $1$-dimensional mean estimation and density estimation without batches (i.e. for $k=1$). In these settings, the authors surprisingly prove that the algorithm that would be used in absence of corruption is automatically robust to Huber contamination.\\

In our setting, we would like to answer the following questions:

\vspace{-2.5mm}

\begin{enumerate}
    \item When contamination comes \textit{after} privacy, do we need to design robust procedures or would the private procedure be automatically robust like in \cite{li2022robustness}?
        
    \vspace{-2.5mm}
    
    \item If $Q_\epsilon$ denotes the optimal privacy mechanism for $\epsilon$-contamination, how does $Q_\epsilon$ depend on $\epsilon$? 
\end{enumerate}

We answer these questions as follows:
    
\vspace{-2.5mm}
    
\begin{enumerate}
    \item With contamination \textit{after} privacy, the procedure that we would use if there were no contamination is no longer robust and a new algorithm is needed.
    
    \vspace{-2.5mm}
    
    \item The optimal privacy mechanism $Q_\epsilon$ does not depend on $\epsilon$, whereas the optimal estimator does.
\end{enumerate}

We introduce the minimax framework. An \textit{estimator} $\widehat{p}$ is a measurable function of the data taking values in~$\mathcal{P}_d$.
$$\widehat{p}: \mathcal{Z}^{nk} \longrightarrow \mathcal{P}_d. $$

For any set of $n'$ clean batches $Y^1,\dots, Y^{n'}$ where $Y^b=(Y_1^b, \dots ,Y_k^b)$ and $n' = n(1-\epsilon)$, we define the set of $\epsilon$-contaminated families of $n$ batches as 
\begin{equation}\label{defC}
    \mathcal{C}(Y^1,\dots, Y^{n'}) = \left\{(Z^b)_{b=1}^n \,\Big|\, \exists J \hspace{-1mm} \subset \hspace{-1mm} [n] \, \text{ s.t. } |J| \hspace{-1mm} = \hspace{-1mm} n\epsilon \text{ and } \{Z^{\, b}\}_{b \notin J} = \{Y^1,\dots, Y^{n'}\}\right\}.
\end{equation}

We are interested in estimating $p \in \mathcal{P}_d$ with guarantees in high probability. We therefore introduce the minimax estimation rate of $p$ in high probability as follows.

\begin{definition}\label{Def_minimax_risk}
Given $\delta>0$, the minimax rate of estimation rate of $p \in \mathcal{P}_d$ given the privatized and $\epsilon$-corrupted batches $(Z^b)_{ b=1}^n$ where $\forall i \in \{1,\dots, n\}: Z^b = (Z_1^b, \dots Z_k^b)$ is defined as the quantity $\psi_{\delta}^*(n,k,\alpha, d, \epsilon)$ satisfying
\begin{equation}\label{def:minimax_rate}
    \psi_\delta^*(n,k,\alpha, d, \epsilon) = \inf \Bigg\{\psi>0 ~\Big|~ \inf_{\hat p} \sup_{p \in \mathcal{P}_d} \mathbb{P}\Big(\sup_{z \in \mathcal{C}(Y)} \big\|\widehat{p}(z)-p\big\|_1 > \psi\Big) \leq \delta\Bigg\}.
\end{equation}
\end{definition}

where the infimum is taken over all estimators $\widehat{p}$ and all $\alpha$-LDP mechanisms $Q$, and the expectation is taken over all collections of $n'$ clean batches $Y^1, \dots, Y^{n'}$ where $Y^b = (Y_1^b,\dots, Y_k^b)$ and $Y_l^b \overset{iid}{\sim} Qp$.
Informally, $\psi^*_\delta$ represents the infimal distance such that there exists an estimator $\hat{p}$ able to estimate any $p \in \mathcal{P}_d$ within total variation $\psi^*_\delta$ with probability $\geq 1-\delta$.
The $\ell_1$ norm is a natural metric for estimating discrete distributions since $TV(p,q) = \frac{1}{2}\|p-q\|_1$ for any $p,q \in \mathcal{P}_d$ (see~\cite{tsybakov2008introduction}).

%% file: contents/Results.tex
\section{Results}
We now state our main Theorem.
\begin{theorem}\label{main_th}
Assume $d \geq 3$. There exist absolute constants $c, C, C', C''>0$ such that for $\delta = C' e^{-d}$ $1-C' e^{-d}$, we have:
$$\psi^*_\delta(n,k,\alpha, \epsilon, d) \geq c\bigg\{\bigg(\frac{d}{\alpha\sqrt{k n}} + \frac{\epsilon}{\alpha} \sqrt{\frac{d}{k}} \bigg) \land 1 \bigg\},$$
and if $n \geq C'' d$ then
$$ \psi^*_\delta(n,k,\alpha, \epsilon, d)\leq C\bigg\{\bigg(\frac{d }{\alpha\sqrt{k n}} + \frac{\epsilon\sqrt{\log(1/\epsilon)}}{\alpha} \sqrt{\frac{d}{k}} \bigg) \land 1\bigg\}.$$
\end{theorem}

In short, we prove that with probability at least $1-O(e^{-d})$, it is possible to estimate any $p \in \mathcal{P}_d$ within total variation of the order of $\bigg(\frac{d}{\alpha\sqrt{k n}} + \frac{\epsilon}{\alpha} \sqrt{\frac{d}{k}} \bigg) \land 1$ up to log factors and provided that $n \geq C'' d$. We can compare this rate with existing results in the literature.
\begin{itemize}
    \item As shown in \cite{duchi2014local}, the term $\frac{d}{\alpha\sqrt{k n}} \land 1$ corresponds to the estimation rate under privacy if there were no outliers, with a total number of observations of $N = nk$.
    \item The term $\frac{\epsilon \sqrt{d}}{\alpha \sqrt{k}}\land 1$ reveals an interesting interplay between contamination and privacy. In absence of privacy, \cite{qiao2017learning} proved that the contribution of the contamination is of the order of $\frac{\epsilon}{\sqrt{k}} \land 1$. The effect of the corruption therefore becomes more dramatic when it occurs after privatization.
    \item Letting $k' = \frac{\alpha^2}{d}k$, our rate rewrites $\psi^*(n,k,\alpha, \epsilon, d) \asymp \Big(\sqrt{\frac{d}{k' n}} + \frac{\epsilon}{ \sqrt{k'}}\Big) \land 1$. 
    Noticeably, this rate exactly corresponds to the rate from \cite{qiao2017learning} if we had an $\epsilon$-corrupted family of $n$ \textit{non-privatized} batches $X_1,\dots, X_n$, and if each batch contained $k'$ observations. 
    The quantity $k'$ therefore acts as an effective sample size and the effect of privacy amounts to shrinking the number of observations by a factor $\alpha^2/d$.
    \item For the upper bound, the assumption $n \geq C'' d$ is classical in the robust statistics literature, even in the gaussian setting (see e.g.~\cite{dalalyan2020all}). 
\end{itemize}

\subsection{Lower bound}

The following Proposition yields an information theoretic lower bound on the best achievable estimation accuracy under local differential privacy and adversarial contamination.

\begin{proposition}\label{LB_Robustness_Privacy}
Assume $d\geq 3$. There exist two absolute constants $C,c>0$ such that for all $\epsilon \in (0,\frac{1}{2})$, for all estimator $\hat p$ and all $\alpha$-LDP mechanism $Q$, there exists a probability vector $p \in \mathcal{P}_d$ satisfying
$$ \mathbb{P}_p\left[\sup_{z' \in \mathcal{C}(Y)} \big\|\widehat{p}(z')-p\big\|_1 \geq c\bigg\{\bigg(\frac{d}{\alpha\sqrt{k n}} + \frac{\epsilon \sqrt{d}}{\alpha \sqrt{k}}\bigg) \land 1\bigg\}\right] \geq Ce^{-d},$$
where the probability $\mathbb{P}_p$ is taken over all collections of $n'$ clean batches $Y = (Y^1, \dots, Y^{n'})$ where $Y^b = (Y_1^b,\dots, Y_k^b)$ and $Y_l^b \overset{iid}{\sim} Qp$.
\end{proposition}

The proof is given in Appendix \ref{sec:proof_LB}. 
At a high level, the term $\frac{d}{\alpha\sqrt{k n}} \land 1$ comes from the classical lower bound given in~\cite{duchi2014local}. 
The proof of the second term $\frac{\epsilon \sqrt{d}}{\alpha \sqrt{k}} \land 1$ is new. 
It is based on the fact that for any $\alpha$-LDP mechanism $Q$, it is possible to find two probability vectors $p,q \in \mathcal{P}_d$ such that $\|p-q\|_1 \gtrsim \frac{\epsilon \sqrt{d}}{\alpha \sqrt{k}} \land 1$ and $TV(Qp^{\otimes k},Qq^{\otimes k}) \leq \epsilon$. 
In other words, we prove:
$$ \inf_Q \hspace{-3mm} \sup_{\substack{
     (p,q)\in \mathcal{P}_d:  \\
      TV(Qp^{\otimes k},Qq^{\otimes k}) \leq \epsilon}} \hspace{-5mm} \|p-q\|_1 \gtrsim \frac{\epsilon \sqrt{d}}{\alpha \sqrt{k}} \land 1.
$$
In the proof, we argue that $TV(Qp^{\otimes k},Qq^{\otimes k}) \leq \epsilon$ represents an indistinguishability condition under $\epsilon$-contamination. Namely, it implies that, even if we had arbitrarily many clean batches drawn from $p$ or $q$, the adversary could add $n\epsilon$ corrupted batches such that the resulting family of batches has the same distribution under $p$ or $q$. 
By observing this limiting distribution, it is therefore impossible to recover the underlying probability distribution so that an error of $\|p-q\|_1/2$ is unavoidable.\\

To exhibit two vectors $p,q \in \mathcal{P}$ satisfying this, we restrict ourselves to vectors satisfying $\chi^2(Qp||Qq) \leq C\frac{\epsilon^2}{k}$ for some small enough absolute constant $C>0$, which implies that $TV(Qp^{\otimes k},Qq^{\otimes k}) \leq \epsilon$ (see~\cite{tsybakov2008introduction} section 2.4).
Noticeably, we prove the relation
$$ \chi^2(p||q) = \Delta^T \Omega \Delta,$$
where $\Delta = p-q$ and $\Omega = \Omega(Q) = \left[{\displaystyle\int_{\mathcal{Z}}} \left(\frac{Q(z|i)}{Q(z|1)} -1\right)\left(\frac{Q(z|j)}{Q(z|1)} -1\right) Q(z|1) dz\right]_{i,j \in [d]} \hspace{-5mm} $ is a nonnegative symmetric matrix. The eigenvectors of $\Omega$ play an important role. Namely, we prove that we can choose a vector $\Delta$ in the span of the first $\lceil\frac{2d}{3} \rceil$ eigenvectors of $\Omega$ such that $\Delta^T \Omega \Delta \leq C \frac{\epsilon^2}{k}$ and $\|\Delta\|_1\gtrsim \frac{\epsilon \sqrt{d}}{\alpha \sqrt{k}} \land 1$. Defining the vectors $p = \left(\frac{|\Delta_j|}{\|\Delta\|_1}\right)_{j=1}^d \in \mathcal{P}_d$ and $q = p-\Delta$ ends the proof.

%% file: contents/UpperBoundSketch.tex
We now address the upper bound by proposing an $\alpha$-LDP mechanism $Q$ for privatizing the clean data $X^1,\dots, X^{n'}$ as well as an algorithm $\widehat{p}$ for robustly estimating vector $p$ given an $\epsilon$-contaminated family of $n$ batches $Z^1,\dots, Z^n$.\\

Each non-private data point $X_i^{b} \in [d]$ is privatized using the \rappor algorithm introduced in (\cite{duchi2014local, kairouz2016discrete}). In this procedure, the \textit{privatization channel} $Q$ randomly maps each point  $X \in [d]$ to a point $Z \in \{0,1\}^d$ by flipping its coordinates independently at random with probability $\lambda= \frac{1}{e^{\alpha/2}+1}$:
$$
\forall j \in [d] : ~~ Z(j) = \begin{cases} \mathds{1}_{X=j}&\text{with probability } 1-\lambda,\\
    1- \mathds{1}_{X=j} &\text{otherwise.}
    \end{cases} ~~ 
$$ 
We now derive a polynomial-time algorithm taking as input the $\epsilon$-contaminated family of batches $(Z^b)_{b \in [n]}$ and returning an estimate $\hat p$ for $p$ with the following properties.

\begin{theorem}[Upper Bound]\label{thm:upperbound}
For any $\epsilon \in (0,1/100]$, $\alpha \in (0,1]$, if  $n \geq \frac{4d}{\epsilon^2\ln(e/\epsilon)}$, Algorithm \ref{alg:estimation_proc} runs in polynomial time in all parameters and its estimate $\widehat{p}$ satisfies $||\widehat{p}-p||_1 \lesssim \frac{\epsilon}{\alpha}\sqrt{\frac{d\ln(1/\epsilon)}{k}}$ w.p. at least $1-O(e^{-d})$.
\end{theorem}

If $n\geq O(d)$, then there exists $\epsilon' \in (0,1/100]$ s.t. $n = \frac{4d}{(\epsilon')^2\ln(1/\epsilon')}$. Running the algorithm with that parameter $\epsilon'$ rather than the true $\epsilon$ gives the following result.
\begin{corollary}\label{cor:rate_priv_alone}
If $n \geq O(d)$, then the algorithm's estimate satisfies $||\widehat{p}-p||_1 \lesssim \frac{d}{\alpha}\sqrt{\frac{e}{nk}}$ with probability at least $1-O(e^{-d})$.
\end{corollary}
Theorem \ref{thm:upperbound} and Corollary \ref{cor:rate_priv_alone} yield the upper bound. 
We have not seen the regime $d \leq n$ explored in the literature, even with robustness only. This would be an interesting research direction for future work. Note that for the estimate $\hat{p}$ given by Algorithm \ref{alg:estimation_proc} we can have $\|\hat p\|_1 \neq 1$. The next corollary, proved in Appendix \ref{proof_cor_actual_estim}, states that normalizing $\hat{p}$ yields an estimator in $\mathcal{P}_d$ with the same estimation guarantees as in Theorem \ref{thm:upperbound}.
\begin{corollary}\label{cor:true_estim}
Let the assumptions of Theorem \ref{thm:upperbound} be satisfied and let $\hat p$ denote the output of Algorithm \ref{alg:estimation_proc}. Define $\widehat p^* = \frac{\widehat p}{\|\hat p\|_1}$, then $||\widehat{p}^*-p||_1 \lesssim \frac{\epsilon}{\alpha}\sqrt{\frac{d\ln(1/\epsilon)}{k}}$ holds with probability at least $1-O(e^{-d})$.
\end{corollary}

\subsection{Description of the algorithm}

We now give a high level description of our algorithm. It is based on algorithms for robust discrete distribution estimation, \cite{jain2020optimal,jain2021robust}. For each $S \subseteq[d]$, define $q(S)=\sum\limits_{j \in S}q_j$ and $p(S)=\sum\limits_{j \in S}p_j$. The quantities $\hatq$, $\widehat{p}$ will respectively denote the estimators of $p$ and $q$. Recalling that $TV(p,\widehat p) = \sup_{S \subseteq [d]}\big|p(S) - \widehat{p}(S)\big|$, we aim at finding $\widehat{p}$ satisfying $\big|p(S) - \widehat{p}(S)\big| \lesssim \frac{\epsilon}{\alpha}\sqrt{\frac{d\ln(1/\epsilon)}{k}}$ for all $S \subseteq[d]$. To this end, it is natural to first estimate the auxiliary quantity
$$ q(j) := \mathds{E}_p\big[Z(j)\big|\,Z \text{ is a good sample}\big] ~~ \text{ for all } j \in [d],$$
which is linked with $p(j)$ through the formula $p(j) = \frac{q(j)-1}{1-2\lambda}$. 
Our algorithm therefore first focuses on robustly estimating $q$ and outputs $\widehat{p} = \frac{\mub-\mathds{1}}{1-2\lambda}$.
If there were no outliers, we would estimate $q(j)$ by $ \frac{1}{nk} \sum_{b \in [n]} \sum_{l \in [k]} Z_l^{b}(j)$. 
In the presence of outliers, our algorithm iteratively deletes the batches that are likely to be contaminated, and returns the empirical mean of the remaining data. 
More precisely, at each iteration, the current collection of remaining batches $B'$ is processed as follows:
\begin{enumerate}
    \item Compute the \textit{contamination rate} $\sqrt{\tau_{B'}}$ (defined in equation \ref{eq:defcontaminationrate}) of the collection $B'$. If $\sqrt{\tau_{B'}}\leq 200$, return the empirical mean of the elements in $B'$.
    \item If $\sqrt{\tau_{B'}}\geq 200$, compute the \textit{corruption score} $\varepsilon_b$ (defined in equation \ref{eq:corruptionscore}) of each batch $b \in B'$. Select the subset $B^o$ of the $n\epsilon$ batches of $B'$ with top corruption scores. Iteratively delete one batch in $B^o$: at each step, choose a batch $b$ with probability proportional to $\varepsilon_b$, until the sum of all $\varepsilon_b$ in $B^o$ has been halved.
\end{enumerate}

At a high level,  \textit{contamination rate} $\tau_{B'}$ quantifies how many adversarial batches remain in the current collection $B'$. The \textit{corruption score} $\varepsilon_b$ quantifies how likely it is for batch~$b$ to be an outlier. Both the \textit{contamination rate} and the \textit{corruption scores} can be computed in polynomial time (see Remark \ref{rem:polynomial_time}). The algorithm therefore terminates in polynomial time, as it removes at least one batch per iteration. We give its pseudo-code below.

\begin{algorithm2e}[ht]\label{alg:estimation_proc}
\DontPrintSemicolon
\SetKwInOut{Input}{input}
\Input{Corruption level $\epsilon$, Batch collection $B$}
$B'\leftarrow B$\;
	\While{ contamination rate of $B', \ \sqrt{\tau_{B'}}\geq 200$}{
	$\forall b \in B'$ compute corruption score $\varepsilon_b$\;
	$B^o\leftarrow\{\text{$\epsilon|B|$ Batches with top corruption scores}\}$\;
	$\epsilon_{\text{tot}}= \sum_{b \in B^o}\varepsilon_b$\;
	\While{$\sum_{b \in B^o}\varepsilon_b \geq \epsilon_{\text{tot}}/2$}{
	Delete a batch from $B^o$, picking batch $b$ with probability proportional to $ \varepsilon_b$
	}
	}
	$\hatq_{B'} = \frac{1}{|B'|} \sum_{b \in B'} \sum_{l=1}^k Z_l^b$ and $\hat{p} = \frac{\mub-\mathds{1}}{1-2\lambda}$\\
\SetKwInOut{Output}{output}
\Output{Estimation $\hat{p}$}
	\caption{\robustsubroutine \label{alg:robustestimator}}
\end{algorithm2e}

We now give a high level description of our algorithm's theoretical guarantees. Recall that $B_G$ denotes the set of non-contaminated batches and $B_A$ the set of adversarial batches. Throughout the paper, for any collection of batches $B'\subseteq[d]$, we will use the following shorthands:
$$
B'_G = B' \cap B_G \text{\ and \ } B'_A= B'\cap B_A.
$$
Assume that $n \geq O\left(\frac{d}{\epsilon^2\log(e/\epsilon)}\right)$.
\begin{itemize}
    \item In Lemma \ref{lem:goodvsbad}, we show that each deletion step has a probability at least $3/4$ of removing an adversarial batch. By a direct Chernoff bound, there is only a probability $\leq O(e^{-\epsilon |B]})\leq O(e^{-d})$ of removing more than $2\epsilon|B_G|$ clean batches before having removed all the corrupted batches. In other words, our algorithm keeps at least $(1-2\epsilon)n$ of the good batches with high probability.
    \item As proved in equations \ref{eq:corruptionscoretoerror} and \ref{eq:boundtauBG}, as soon as a subset $B'$ contains at least $(1-2\epsilon)n$ good batches, it holds with probability $\geq 1 - O(e^{-d})$ that for all $S \subseteq[d]$
    \begin{align}\label{eq:contaminationtoerror}
    \begin{cases}
        |\hatq_{B'}(S)-q(S)| \lesssim  \left(1+\sqrt{\tau_{B'}}\right)\epsilon\sqrt{\frac{d\ln(e/\epsilon)}{k}}, & \text{(i)}\\
    \sqrt{\tau_{B'_G}} \leq 200. & \text{(ii)}
    \end{cases}
    \end{align}
    There are two cases. If the algorithm has eliminated all the outliers, then it has kept at least $(1-2\epsilon)n$ clean batches with probability $1-O(e^{-d})$. Then condition (i) $\sqrt{\tau_{B'}} = \sqrt{\tau_{B'_G}} \leq 200$ ensures that the algorithm terminates. 
    Otherwise, the algorithm stops before removing all of the outliers, but in this case, the termination condition guarantees that $\sqrt{\tau_{B'}} \leq 200$. In both cases, condition (ii) yields that the associated estimator $\hatq := \hatq_{B_{\textbf{out}}}$ has an estimation error satisfying $\sup\limits_{S \subseteq[d]} |\hatq(S) - q(S)| \lesssim \epsilon\sqrt{\frac{d\ln(e/\epsilon)}{k}}$ with probability $\geq 1 - O(e^{-d})$.
    \item Finally, we link the estimation error of $\hatq$ to that of $\hat p$
    \begin{align*}
    ||\widehat{p}-p||_1 & \leq 2\max_{S \in [d]} |\widehat{p}(S)-p(S)|~~~~ \text{ (see Lemma \ref{lem:norm_equal_sup_sets})}\\
    &\leq2\max_{S \subseteq [d]} \Big|\sum_{j \in S}\frac{1}{1-2 \lambda}\left(\hat{q}_j-1\right)-\frac{1}{1-2 \lambda}\left(q_j-1\right)\Big|\\
    &\leq \frac{1}{1-2 \lambda}\max_{S \in [d]} |\hat{q}(S)-q(S)|\leq \frac{5}{\alpha}\max_{S \in [d]} |\hat{q}(S)-q(S)|\\
    & \lesssim \frac{\epsilon}{\alpha}\sqrt{\frac{d\ln(e/\epsilon)}{k}} ~~ \text{ with probability } \geq 1 - O(e^{-d}),
\end{align*}
which yields the estimation guarantee over $\widehat{p}$ and proves Theorem \ref{thm:upperbound}.
\end{itemize}

We now move to the formal definitions of the quantities involved in the algorithm and state all the technical results mentioned.

\subsection{Technical results}
\textit{Wlog}, assume that $6\epsilon\sqrt{\frac{d\ln(e/\epsilon)}{k}}\leq1$. Otherwise the upper bound of the theorem is clear. For any set $S \subseteq[d]$ and any observation $Z^b_i$, we define the empirical weight of $S$ in $Z^b_i$ as $Z^b_i(S) := \sum\limits_{j \in S} Z^b_i(j)$. This quantity is an estimator of $q(S)$. For each batch $Z^b$ and each collection of batches $B' \subseteq B$, we aggregate these estimators by building
$$
\mub_b(S) := \frac{1}{k}\sum_{i=1}^k Z^{b}_i(S) \text{\ \ and \ \ } \hatq_{B'}(S):= \frac{1}{|B'|}\sum_{b\in B'}\mub_b(S).
$$
Our goal is to remove batches $Z^b$ that do not satisfy some concentration properties verified by clean batches. To this end, we introduce empirical estimators of the second order moment:
\begin{align}
    \widehat{\text{Cov}}^{B'}_{S,S'}\left(b\right) &:= \Big[\mub_b(S)-\hatq_{B'}(S)\Big]\Big[\mub_b(S')-\hatq_{B'}(S')\Big] \label{def_cov_b}\\
    \widehat{\text{Cov}}_{S,S'}\left(B'\right)&:= \frac{1}{|B'|}\sum_{b\in B'} \widehat{\text{Cov}}^{B'}_{S,S'}\left(b\right). \label{def_cov_B'}
\end{align}

In Appendix \ref{app:matrixexpression}, we give the expression of matrix $\cov_{S,S'}\left(q\right)$ s.t.
$$
\text{Cov}_{S,S'}\left(q\right)
= \mathbb{E}\left[\widehat{\text{Cov}}_{S,S'}\left(B'\right)\right].$$

We are now ready to define the essential concentration properties satisfied by the clean batches with high probability (see Lemma \ref{lem:essentialproperties}).


 \begin{definition}[Nice properties of good batches]\label{def:nice_properties}\\
1. For all $S \subseteq [d]$, all sub-collections $B_{G}^{\prime} \subseteq B_{G}$ of good batches of size $\left|B_{G}^{\prime}\right| \geq(1-2\epsilon  )\left|B_{G}\right|$,
\begin{align}
    \left|\hatq_{B_{G}^{\prime}}(S)-q(S)\right| &\leq 6\epsilon \sqrt{\frac{d\ln ( e / \epsilon)}{k}}, \label{concent_first_moment}\\
    \left|\widehat{\text{Cov}}_{S,S'}\left(B'_G\right)-\text{Cov}_{S,S'}(\hatq_{B'_G})\right| &\leq \frac{250 d\epsilon \ln \left(\frac{ e}{\epsilon}\right)}{k}. \label{concent_second_moment}
\end{align}

2. For all $S, S' \subseteq [d]$, for any sub collection of good batches $B_G^{''}$ s.t. $|B_G^{''}|\leq \epsilon |B_G|$, 
$$
\sum_{b\in B_G^{''}}\Big[\mub_b(S)-q(S)\Big]\Big[\mub_b(S')-q(S')\Big]\leq \frac{33\epsilon d |B_G| \ln ( e / \epsilon)}{k}.
$$
\end{definition}

\begin{lemma}[Nice properties of good batches]\label{lem:essentialproperties}
If $|B_G|\geq \frac{3d}{\epsilon^2\ln(e/\epsilon)}$, the nice properties of the good batches hold with probability $1-10e^{-d}$.
\end{lemma}

The proof is very similar to that of Lemma 3 in \cite{jain2020optimal}, and can be found in Appendix \ref{app:essentialproperties} where we clarify which technical elements change. \\

In the case where $S'=S$, we use the shorthands $\hatcov_{S,S}(B')=\hatvar_{S}(B')$ and $\cov_{S,S}(B')=\var_{S}(B')$. The following Lemma states that the quality of estimator $\hatq_{B'}$ is controlled by the concentration of $\left|\hatvar_{S}(B')-\var_{S}(\hatq)\right|$. 

\begin{lemma}[Variance gap to estimation error]\label{lem:vartoerror}
If conditions 1 and 2 hold and $\max_{S \subseteq [d]}|\hatq_{B'}(S)-q(S)|\leq 11$, then for any subset $B'$ s.t. $|B'_G|\geq (1-2\epsilon)|B_G|$ and for any $S\subset[d]$, we have:
$$
|\hatq_{B'}(S)-q(S)|\leq 28\epsilon \sqrt{\frac{d\ln(6e/\epsilon)}{k}}+2\sqrt{\epsilon\left|\hatvar_{S}(B')-\var_{S}(\hatq)\right|}.
$$
\end{lemma}

This Lemma is proved in Appendix \ref{app:vartoerror}. Together with equation \eqref{concent_second_moment}, this Lemma ensures that removing enough outliers yields an estimator $\hatq_{B'}$ with estimation guarantee $\sup\limits_{S \subseteq[d]} |\hatq_{B'}(S)-q(S)| \lesssim \epsilon \sqrt{\frac{d\ln ( 1 / \epsilon)}{k}}$.\\

The adversarial batch deletion is achieved by identifying the batches $Z^b$ for which $\widehat{\text{Cov}}^{B'}_{S,S'}\left(b\right)$ (defined in equation \eqref{def_cov_b}) is at odds with Definition \ref{def:nice_properties} for some $S, S'\subset[d]$. 
Searching through all possible $S, S' \subseteq[d]$ would yield an exponential-time algorithm. 
A way around this is to introduce a semi-definite program that can be approximated in polynomial time. 
To this end, we prove the next Lemma, stating that the quantities $\hatcov_{S,S'}\left(q\right)$ and  $\cov_{S,S'}\left(q\right)$ can be computed as scalar products of matrices.
\begin{lemma}[Matrix expression]\label{lem:matrixexpression}
Denote by $\mathds{1}_S$ the indicator vector of the elements in $S$. For each vector $q$, there exists a matrix $\C(\hatq)$ s.t. for any $S,S'\subseteq [d]$,
$$
\cov_{S,S'}\left(\hatq\right)= \left<\indS{S}\indS{S'}^T,\C(\hatq)\right>.
$$
$$
\widehat{\text{Cov}}^{B'}_{S,S'}\left(b\right)= \left<\indS{S}\indS{S'}^T,\hatC_{b,B'}\right> ~~ \text{ and } ~~ \widehat{\text{Cov}}_{S,S'}\left(B'\right)= \left<\indS{S}\indS{S'}^T,\hatC (B')\right>,
$$
with $\hatC (B')= \sum_{b \in B'}\hatC_{b,B'}$.
\end{lemma}

The proof of the Lemma and the precise expressions of the matrices can be found in Appendix \ref{app:matrixexpression}. To define the semi-definite program, we introduce the following space of Gram matrices:
$$
\G:= \left\{M \in \mathbb{R}^{d \times d}, M_{ij}=\langle u^{(i)},v^{(j)}\rangle \;\Big| \; (u^{(i)})_{i=1}^d,(v^{(i)})_{j=1}^d \text{unit vectors in $(\mathbb{R}^{d},\|\cdot\|_2)$}\right\}.
$$

For a subset $B'$, let us define $D_{B'}=\hatC (B')-\C(\hatq_{B'})$, and define $M^*_{B'}$ as any matrix s.t.
$$
\langle M^*_{B'},D_{B'}\rangle \geq \max\limits_{M \in \mathcal{G}} \langle M,D_{B'}\rangle - c\frac{\epsilon d \ln(e/\epsilon)}{k},
$$
for some small enough absolute constant $c>0$. 

\begin{remark}\label{rem:polynomial_time}
Note that the quantity $\max\limits_{M \in \mathcal{G}} \langle M,D_{B'}\rangle$ is an SDP. For all desired precision $\delta>0$, it is possible to find the solution of this program up to an additive constant $\delta$ in polynomial time in all the parameters of the program and in $\log(1/\delta)$. Thus, $M^*_{B'}$ can be computed in polynomial time, as well as the contamination rate and the corruption score, defined below.
\end{remark}

\paragraph{Definition of the \textit{contamination rate} and \textit{corruption scores}.}
When $\hatq(S)\gg \lambda |S|$ for some $S \subseteq [d]$, the \textit{contamination rate} and \textit{corruption scores} have  special definitions. Formally, let $A = \left\{j \in [d] \,\big| \, \hatq_{B'}(j) \geq \lambda\right\}$ and $S^*=\max_{S \subseteq [d]}\left|\hatq_{B'}(S)-\lambda |S|\right|$. We have $S^* = A$ or $S^* = [d]\setminus A$, which can be computed in polynomial time. In the special case where $\left|\hatq_{B'}(S^*)-\lambda |S^*|\right|\geq 11$, the \textit{contamination rate} $\sqrt{\tau_{B'}}$ of the collection $B'$ is defined as $\tau_{B'} = \infty$ and the corruption score of a batch is defined as $\varepsilon_b(B')= \left|\mub_{b}(S^*)-\lambda |S^*|\right|$. \\

Otherwise, the \textit{contamination rate} $\sqrt{\tau_{B'}}$ of the collection $B'$ is defined through the quantity satisfying
\begin{equation}\label{eq:defcontaminationrate}
    \langle M^*_{B'},D_{B'}\rangle=\tau_{B'} \frac{\epsilon d \ln(e/\epsilon)}{k}.
\end{equation}

Define the \textit{corruption score} of a batch as
\begin{equation}\label{eq:corruptionscore}
    \varepsilon_b(B')= \langle M^*_{B'},\hatC_{b,B'}\rangle.
\end{equation}

The following Lemma guarantees that the quantity $\langle M^*_{B'},D_{B'}\rangle$ is a good approximation of $\max\limits_{S,S' \subseteq[d]} \big|\langle \indS{S}\indS{S'}^T,D_{B'} \rangle\big|$, with the advantage that it can be computed in polynomial time.

\begin{lemma}[Grothendieck's inequality corollary]\label{lem:Groetendickapprox}
Assume $d\geq 3$. For all symmetric matrix $A \in \R^{d \times d}$, it holds $$
\max_{S,S' \subseteq[d]} \big|\langle \indS{S}\indS{S'}^T,A \rangle\big| \leq \max_{M \in \mathcal{G}} \langle M,A \rangle\leq 8\max_{S,S' \subseteq[d]}\big|\langle \indS{S}\indS{S'}^T,A \rangle\big|.
$$
\end{lemma}
The proof of the Lemma can be found in Appendix \ref{app:Groetendickapprox}. Together with Lemma \ref{lem:vartoerror}, this Lemma implies that if conditions 1 and 2 hold, then for any subset $B'$ s.t. $|B'_G|\geq (1-2\epsilon)|B_G|$ and for any $S\subset[d]$, we have:
\begin{equation}\label{eq:corruptionscoretoerror}
    |\hatq_{B'}(S)-q(S)|\leq \left(30+2 \sqrt{\tau_{B'}}\right)\epsilon \sqrt{\frac{d\ln(e/\epsilon)}{k}}.
\end{equation}

This Lemma implies that if equation \eqref{concent_second_moment} holds, then, for any $B'$ s.t. $|B'_G|\geq (1-2\epsilon)|B_G|$
\begin{equation}\label{eq:boundtauBG}
    \sqrt{\tau_{B'_G}}\leq 200.
\end{equation}

\begin{lemma}[Score good vs. adversarial batches]\label{lem:goodvsbad}
If $\sqrt{\tau_{B'}}\geq 200$ and condition 1-2 hold, then for any collection of batches $B'$ s.t. $|B'\cap B_G|\geq (1-2 \epsilon)|B_G|$, for any sub-collection of good batches $B^{''}_G\subseteq B$, $|B^{''}_G|\leq \epsilon n$, we have:
$$
\sum_{b \in B^{''}_G}\varepsilon_b(B') < \frac{1}{8} \sum_{b \in B^{'}_A}\varepsilon_b(B').
$$
\end{lemma}
This Lemma is proved in Appendix \ref{app:goodvsbad}, where we argue that this Lemma ensures that each batch deletion has a probability at least $\frac{3}{4}$ of removing an adversarial batch.

%% file: contents/Discussion_and_future_work.tex
\section{Discussion and future work}

We studied the problem of estimating discrete distributions in total variation, with both privacy and robustness constraints. We obtained an information theoretic lower bound of $\epsilon\sqrt{d/\alpha^2 k}+\sqrt{d^2/\alpha^2 kn}$. We  proposed an algorithm running in polynomial time and returning an estimated parameter such that the estimation error is within $\sqrt{\log(1/\epsilon)}$ of the information theoretic lower bound. It would be interesting to explore if polynomial algorithms could achieve the optimal bound without this extra factor. 
We do not consider the adaptation to unknown contamination $\epsilon$ and leave it for future work. 
It would also be interesting to explore what happens if the contamination occurs before the privacy rather than after, like in ~\cite{li2022robustness}. Indeed, 
they do not consider batched data, and it would be interesting to check if their result holds in that case.
Also, the upper bound holds only if $n\geq O(d)$. Exploring the regime $n\leq d$ would be an interesting research direction, which has not been done to our knowledge, even in the case of the sole robustness constraint. Finally, we could study the combination of the robustness and privacy constraints in other settings, such as density estimation.

%% file: contents/JainOrlinskyAppendix.tex
\subsection{Proof of Lemma \ref{lem:eqlaw}, Law of the sum}\label{app:eqlaw}
\begin{lemma}[Law of the sum]\label{lem:eqlaw}
For any subset $S \subseteq [d]$, we have:
$$
\sum_{j \in S}Z(j) \sim \sum_{j=1}^{|S|-1}b_j+b^S,
$$
with the $(b_j)_{j=1}^{|S|-1}$ independent Bernoulli variables s.t. $\mathrm{P}(b_j=1)=\lambda$ and $b^S$ a Bernoulli independent of the others s.t.
$$
\mathrm{P}(b^S=1)=\lambda+(1-2\lambda)p_S.
$$
\end{lemma}
For any $t \in [d]$,
\begin{align*}
    \mathbb{P}\left(\sum_{j \in S}Z(j)=t\right)&= \binom{|S|-1}{t-1}\left(1-\lambda\right)^{|S|-t+1}\lambda^{t-1}  p(S)+\left(1-p(S)\right)\binom{|S|}{t}\left(1-\lambda\right)^{|S|-t}\lambda^{t} \\
    &+\binom{|S|-1}{t}\left(1-\lambda\right)^{|S|-t-1}\lambda^{t+1}  p(S)\\
    &= \binom{|S|-1}{t-1}\left(1-\lambda\right)^{|S|-t}\lambda^{t-1}\left[ (1-\lambda) p(S)+\lambda\left(1-p(S)\right)\right]\\
    &+\binom{|S|-1}{t}\left(1-\lambda\right)^{|S|-t-1}\lambda^{t}\left[(1-\lambda)\left(1-p(S)\right)+\lambda p(S)\right]\\
    &= \binom{|S|-1}{t-1}\left(1-\lambda\right)^{|S|-t}\lambda^{t-1}\left[ (1-2\lambda) p(S)+\lambda\right]\\
    &+\binom{|S|-1}{t}\left(1-\lambda\right)^{|S|-t-1}\lambda^{t}\left[1-\lambda-(1-2\lambda) p(S)\right].
\end{align*}
Note that we have:
\begin{align}\label{eq:qS}
   q(S)&=\left(1-2\lambda\right)p(S)+\lambda |S|.
\end{align}

\subsection{Proof of Lemma \ref{lem:essentialproperties}, Essential properties of good batches}\label{app:essentialproperties}
We start with the following intermediary Lemma.
\begin{lemma}\label{lem:meangoodbatches}
If $\left|B_{G}\right| \geq \frac{2 d}{\epsilon^{2} \ln (e/ \epsilon)}$, then $\forall S \subseteq[d]$ and $\forall B_{G}^{\prime} \subseteq B_{G}$ of size $\left|B_{G}^{\prime}\right| \geq(1-2\epsilon)\left|B_{G}\right|$, with probability at least $ 1-4 e^{-d}$,
$$
\left|\hatq_{B_{G}^{\prime}}(S)-q(S)\right| \leq 6\epsilon \sqrt{\frac{d\ln (e/ \epsilon)}{k}}.
$$
\end{lemma}
\begin{proof}: The proof of this lemma is exactly part of that of lemma 11 in \cite{jain2020optimal} with different constants, we repeat it for completeness. From Hoeffding's inequality, for any $S \subseteq[d]$,
$$
\mathbb{P}\left[|B_G|\left|\hatq_{B_{G}}(S)-q(S)\right| \geq  \frac{|B_G|}{\sqrt{2}}\epsilon\sqrt{\frac{d\ln (e/ \epsilon)}{k}}\right] \leq 2 e^{-\epsilon^2|B_G| \ln (e/ \epsilon)} \leq 2e^{-2d} .
$$
Similarly, for a fixed sub-collection $U_{G} \subseteq B_{G}$ of size $1 \leq\left|U_{G}\right| \leq 2\epsilon\left|B_{G}\right|$,
\begin{equation}\label{eq:meansmallcollection}
   \mathbb{P}\left[\left|U_{G}\right| \cdot\left|\hatq_{U_{G}}(S)-q(S)\right| \geq 2 \epsilon|B_G|\sqrt{\frac{d\ln (e/ \epsilon)}{k}}\right]\leq2 e^{-8\frac{\epsilon^2 |B_G|^2 }{|U_G|}\ln (e/ \epsilon)}\leq 2 e^{- 4\epsilon|B_G| \ln (e/ \epsilon)}. 
\end{equation}

We now bound the number of subsets of cardinality smaller than $2\epsilon |B_G|$:
\begin{align}
    \sum_{j=1}^{\left\lfloor2\epsilon\left|B_{G}\right|\right\rfloor}\left(\begin{array}{c}\left|B_{G}\right| \\ j\end{array}\right) &\leq 2\epsilon\left|B_{G}\right|\left(\begin{array}{c}\left|B_{G}\right| \\ \left\lfloor2\epsilon\left|B_{G}\right|\right\rfloor\end{array}\right) \leq 2\epsilon\left|B_{G}\right|\left(\frac{e\left|B_{G}\right|}{2\epsilon\left|B_{G}\right|}\right)^{2\epsilon\left|B_{G}\right|} \nonumber \\
    &\leq e^{2\epsilon\left|B_{G}\right| \ln (e/ \epsilon)+\ln \left(2\epsilon\left|B_{G}\right|\right)}<e^{3 \epsilon\left|B_{G}\right| \ln (e/ \epsilon)}.
\end{align}

Thus, by union bound, 
$$
\mathbb{P}\left[\exists |U_G|\leq2\epsilon |B_G|, \left|U_{G}\right| \cdot\left|\hatq_{U_{G}}(S)-q(S)\right| \geq 2 \epsilon|B_G|\sqrt{\frac{\ln (e /2 \epsilon)}{n}}\right] \leq 2 e^{-\epsilon|B_G| \ln (e/ \epsilon)} \leq 2e^{-2d} .
$$

For any sub-collection $B_{G}^{\prime} \subseteq B_{G}$ with $\left|B_{G}^{\prime}\right| \geq(1-2\epsilon)\left|B_{G}\right|$,
$$
\begin{aligned}
\left|\sum_{b \in B_{G}^{\prime}}\left(\mub_{b}(S)-q(S)\right)\right| &=\left|\sum_{b \in B_{G}}\left(\mub_{b}(S)-q(S)\right)-\sum_{b \in B_{G} \setminus B_{G}^{\prime}}\left(\mub_{b}(S)-q(S)\right)\right| \\
& \leq\left|\sum_{b \in B_{G}}\left(\mub_{b}(S)-q(S)\right)\right|+\left|\sum_{b \in B_{G} \setminus B_{G}^{\prime}}\left(\mub_{b}(S)-q(S)\right)\right| \\
& \leq\left|B_{G}\right| \times\left|\hatq_{B_{G}}(S)-q(S)\right|+\max _{U_{G}:\left|U_{G}\right| \leq 2\epsilon\left|B_{G}\right|}\left|U_{G}\right| \times\left|\hatq_{U_{G}}(S)-q(S)\right| \\
& \leq (2+\frac{1}{\sqrt{2}}) \epsilon\left|B_{G}\right| \sqrt{\frac{\ln (e/ \epsilon)}{n}}.
\end{aligned}
$$
where the last inequality holds with probability at least $1-4e^{-2d}$. We conclude by using a union bound over the $2^d$ possible subsets and by noting that $(2+\frac{1}{\sqrt{2}})\frac{|B_G|}{|B'_G|}\leq 6$. \end{proof}

We now move to the following result.
\begin{lemma}\label{lem:covgoodbatches}
If $\left|B_{G}\right| \geq \frac{3d}{\epsilon^{2} \ln (e/ \epsilon)}$, then $\forall S,S' \subseteq[d]$ and $\forall B_{G}^{\prime} \subseteq B_{G}$ of size $\left|B_{G}^{\prime}\right| \geq(1-2\epsilon)\left|B_{G}\right|$, with probability at least $ 1-2 e^{-d}$,
$$
\left|\frac{1}{|B'_G|}\sum_{b \in B'_G}\left(\mub_{b}(S)-q(S)\right)\left(\mub_{b}(S')-q(S')\right)-\cov_{S,S'}(q)\right| \leq \frac{140 d\epsilon \ln \left(\frac{ e}{\epsilon}\right)}{k}.
$$
\end{lemma}

\begin{proof}: Let $U_{b}(S,S')=\left(\frac{\mub_{b}(S)-q(S)}{d}\right)\left(\frac{\mub_{b}(S')-q(S')}{d}\right)-\frac{\cov_{S,S'}(q)}{d^2}$. For $b \in B_{G}, \frac{\mub_{b}(S)-q(S)}{d}\sim \operatorname{subG}(1 / 4 dk)$, therefore
$$
\left(\frac{\mub_{b}(S)-q(S)}{d}\right)\left(\frac{\mub_{b}(S')-q(S')}{d}\right)-\mathbb{E}\left[\left(\frac{\mub_{b}(S)-q(S)}{d}\right)\left(\frac{\mub_{b}(S')-q(S')}{d}\right)\right]=Y_{b} \sim \operatorname{subE}\left(\frac{16}{4 kd}\right) .
$$

Here subE is sub exponential distribution. For any $S,S' \subseteq[d]$, Bernstein's inequality gives:
$$
\begin{aligned}
\mathbb{P}\bigg[\bigg|\sum_{b \in B_{G}} U_{b}(S,S')\bigg| \geq 6 \epsilon\left|B_{G}\right| \frac{\ln (e/ \epsilon)}{kd}\bigg] & \leq 2 e^{- \epsilon^2\left|B_{G}\right| \ln^2 (e / \epsilon)}\leq 2 e^{-3d}.\\
\end{aligned}
$$
Next, for a fixed sub-collection $B^{''}_{G} \subseteq B_{G}$ of size $1 \leq\left|B^{''}_{G} \right| \leq \epsilon\left|B_{G}\right|$,
$$
\begin{aligned}
\operatorname{Pr}\left[\left|\sum_{b \in B^{''}_{G} } U_{b}(S,S')\right| \geq 64 \epsilon\left|B_{G}\right| \frac{\ln (e / \epsilon)}{n}\right] & \leq 2 e^{-\frac{64 \epsilon\left|B_{G}\right| \ln (e / \epsilon)}{2\times 2 \times 4 / n}} \\
& \leq 2 e^{-4 \epsilon\left|B_{G}\right| \ln (e / \epsilon)}.
\end{aligned}
$$

The same steps as the previous lemma terminate the proof, except that there are now $2^{2d}$ sets $S,S' \subseteq [d]$.

\end{proof}

By Lemma \ref{lem:meangoodbatches} and \ref{lem:covgoodbatches}, if $\left|B_{G}\right| \geq \frac{2 d}{\epsilon^{2} \ln (e/ \epsilon)}$, then $\forall S \subseteq[d]$ and $\forall B_{G}^{\prime} \subseteq B_{G}$ of size $\left|B_{G}^{\prime}\right| \geq(1-2\epsilon)\left|B_{G}\right|$, with probability at least $ 1-8 e^{-d}$:
\begin{align*}
    \left|\hatq_{B_{G}^{\prime}}(S)-q(S)\right| \leq 6\epsilon \sqrt{\frac{d\ln (e/ \epsilon)}{k}}&\\
   & \text{ and }\\
    \left|\frac{1}{|B'_G|}\sum_{b \in B'_G}\left(\mub_{b}(S)-q(S)\right)\left(\mub_{b}(S')-q(S')\right)-\cov_{S,S'}(q)\right| &\leq \frac{140 d\epsilon \ln \left(\frac{6 e}{\beta}\right)}{k}.
\end{align*}
Additionally Lemma \ref{lem:C_lip}, this implies:
$$
\left|\cov_{S,S'}(q)-\cov_{S,S'}(\hatq_{B'_G})\right| \leq 66\epsilon \sqrt{\frac{d\ln (e/ \epsilon)}{k}}.
$$
Moreover:
\begin{align*}
  \frac{1}{|B'_G|}\sum_{b \in B'_G}\left(\mub_{b}(S)-q(S)\right)\left(\mub_{b}(S')-q(S')\right) =& \frac{1}{|B'_G|}\sum_{b \in B'_G}\left(\mub_{b}(S)-\hatq_{B'}(S)\right)\left(\mub_{b}(S')-\hatq_{B'}(S')\right)\\
  &+\left(q(S)-\hatq_{B'}(S)\right)\left(q(S')-\hatq_{B'}(S')\right)\\
 & +\frac{1}{|B'_G|}\sum_{b \in B'_G}\left(\mub_{b}(S)-\hatq_{B'}(S)\right)\left(q(S')-\hatq_{B'}(S')\right)\\
 &+\frac{1}{|B'_G|}\sum_{b \in B'_G}\left(q(S)-\hatq_{B'}(S)\right)\left(\mub_{b}(S')-\hatq_{B'}(S')\right)\\
  &=\frac{1}{|B'_G|}\sum_{b \in B'_G}\left(\mub_{b}(S)-\hatq_{B'}(S)\right)\left(\mub_{b}(S')-\hatq_{B'}(S')\right)\\
  &+\left(q(S)-\hatq_{B'}(S)\right)\left(q(S')-\hatq_{B'}(S')\right).
\end{align*}
Therefore:
$$
\left|\widehat{\cov}_{S,S'}\left(B'_G\right)-\cov_{S,S'}(\hatq_{B'_G})\right| \leq \frac{242 d\epsilon \ln \left(\frac{ e}{\beta}\right)}{k}.
$$
Note that we also have:
\begin{equation}\label{eq:boundcovgoodbatchesbis}
    \left|\widehat{\cov}_{S,S'}\left(B'_G\right)-\cov_{S,S'}(q)\right| \leq \frac{176 d\epsilon \ln \left(\frac{ e}{\epsilon}\right)}{k}.
\end{equation}

The following Lemma gives condition 2.
\begin{lemma}
If $\left|B_{G}\right| \geq \frac{3d}{\epsilon^{2} \ln (e/ \epsilon)}$, then $\forall S,S' \subseteq[d]$ and $\forall B_{G}^{''} \subseteq B_{G}$ of size $\left|B_{G}^{''}\right| \leq\epsilon\left|B_{G}\right|$, with probability at least $ 1-2 e^{-d}$,
$$
\left|\sum_{b\in B_G^{''}}\Big[\mub_b(S)-q(S)\Big]\Big[\mub_b(S')-q(S')\Big]\right| \leq \frac{33\epsilon d |B_G| \ln ( e / \epsilon)}{k}.
$$
\end{lemma}
\begin{proof}:
For any $S,S' \subseteq[d]$ and any $B_{G}^{\prime} \subseteq B_{G}$ Bernstein's inequality gives:
$$
\begin{aligned}
\mathbb{P}\bigg[\bigg|\sum_{b \in B^{''}_{G}} U_{b}(S,S')\bigg| \geq 32 \epsilon\left|B_{G}\right| \frac{\ln (e/ \epsilon)}{kd}\bigg] & \leq 2 e^{-4\epsilon\left|B_{G}\right| \ln (e/ \epsilon)}\\
\end{aligned}.
$$
We have :
\begin{align*}
    \left|\sum_{b\in B_G^{''}}\Big[\mub_b(S)-q(S)\Big]\Big[\mub_b(S')-q(S')\Big]\right| &= \left|\sum_{b \in B^{''}_{G}} d^2U_{b}(S,S')+|B^{''}_G|\cov_{S,S'}(q)\right|\\
    &\leq \left|\sum_{b \in B^{''}_{G}} d^2U_{b}(S,S')\right|+\epsilon\frac{d|B_G|}{k}.
\end{align*}

A union bound over all the possible $B^{''}_G$ and the $2^{2d}$ sets $S,S'$ terminates the proof.

\end{proof}

Combining the three Lemmas of the section gives Lemma \ref{lem:essentialproperties}.

\subsection{Proof of Lemma \ref{lem:vartoerror}, Variance gap to estimation error}\label{app:vartoerror}

\textit{Proof}: By condition 1 and Cauchy-Schwartz:
\begin{align}\label{eq:decobound}
   \left| \hatq_{B'}(S)-q(S)\right|& \leq \frac{1}{|B'|}\left|\sum_{b \in B'_G}\mub_b(S)-q(S)\right|+\frac{1}{|B'|}\left|\sum_{b \in B'_A}\mub_b(S)-q(S)\right|\nonumber\\
   & \leq 6\epsilon \sqrt{\frac{d\ln (e/ \epsilon)}{k}}+ \sqrt{\frac{|B'_A|}{|B'|}}\sqrt{\frac{1}{|B'|}\sum_{b \in B'_A}\Big[\mub_b(S)-q(S)\Big]^2}.
\end{align}

We can decompose the second term:
\begin{align*}
  \frac{1}{|B'|}\sum_{b \in B'_A}\Big[\mub_b(S)-q(S)\Big]^2& =\frac{1}{|B'|}\sum_{b \in B'}\Big[\mub_b(S)-q(S)\Big]^2-\frac{1}{|B'|}\sum_{b \in B'_G}\Big[\mub_b(S)-q(S)\Big]^2.
\end{align*}
By Lemma \ref{lem:covgoodbatches},
$$
\bigg|\frac{1}{|B'_G|}\sum_{b \in B'_G}\Big[\mub_b(S)-q(S)\Big]^2-\var_S(q)\bigg|\leq 140\frac{\epsilon d\ln (e/ \epsilon)}{k}.
$$
Thus, 
\begin{align*}
    \frac{1}{|B'|}\sum_{b \in B'_G}\Big[\mub_b(S)-q(S)\Big]^2& = \frac{|B'_G|}{|B'|}\frac{1}{|B'_G|}\sum_{b \in B'_G}\Big[\mub_b(S)-q(S)\Big]^2\\
    &\geq (1-2\epsilon) \left(\var_S(q)-140\frac{\epsilon d\ln (e/ \epsilon)}{k}\right)\\
    &\geq \var_S(q)-2\epsilon\var_S(q)-140\frac{\epsilon d\ln (e/ \epsilon)}{k}\\
    &\geq \var_S(\hatq_{B'})-15\frac{|\hatq_{B'}(S)-q(S)|}{k}-142\frac{\epsilon d\ln (e/ \epsilon)}{k},
\end{align*}
where the last inequality comes from Lemma \ref{lem:C_lip} and $\var_S(q)\leq d/k$.
Now, we have
\begin{align*}
    |\hatq_{B'}(S)-\hatq_{B'_G}(S)|&\leq \Bigg|\left(\frac{1}{|B'_G|}-\frac{1}{|B'|}\right)\sum_{b \in B'_G}\mub_b(S)\Bigg|+\Bigg|\frac{1}{|B'|}\sum_{b \in B'\setminus B'_G}\mub_b(S)\Bigg|\\
    &\leq \frac{2d\epsilon}{1-\epsilon}\leq 3d\epsilon.
\end{align*}

Thus,
\begin{align*}
    \frac{|\hatq_{B'}(S)-q(S)|}{k}&\leq\frac{|\hatq_{B'_G}(S)-q(S)|}{k}+\frac{|\hatq_{B'}(S)-\hatq_{B'_G}(S)|}{k}\\
    &\leq \frac{6\epsilon}{k} \sqrt{\frac{d\ln (e/ \epsilon)}{k}}+\frac{3d\epsilon}{k}\leq 3\frac{d\epsilon}{k}\ln (e/ \epsilon).
\end{align*}
This implies
\begin{equation}\label{eq:BGtoVar}
    \frac{1}{|B'|}\sum_{b \in B'_G}\Big[\mub_b(S)-q(S)\Big]^2
    \geq \var_S(\hatq_{B'})-187\frac{\epsilon d\ln (e/ \epsilon)}{k}.
\end{equation}

On the other hand,
\begin{align*}
    \frac{1}{|B'|}\sum_{b \in B'}\Big[\mub_b(S)-q(S)\Big]^2=& ~\frac{1}{|B'|}\sum_{b \in B'}\Big[\mub_b(S)-\hatq_{B'}(S)\Big]^2\\ &+\Big[q(S)-\hatq_{B'}(S)\Big]^2+2\Big[q(S)-\hatq_{B'}(S)\Big]\frac{1}{|B'|}\sum_{b \in B'}\Big[\mub_b(S)-\hatq_{B'}(S)\Big]\\
    =& ~\frac{1}{|B'|}\sum_{b \in B'}\Big[\mub_b(S)-\hatq_{B'}(S)\Big]^2 +\Big[q(S)-\hatq_{B'}(S)\Big]^2.
\end{align*}

Combining this equation with equations \ref{eq:BGtoVar} and \ref{eq:decobound} gives

\begin{align*}
   \left| \hatq_{B'}(S)-q(S)\right|& \leq \frac{1}{|B'|}\left|\sum_{b \in B'_G}\mub_b(S)-q(S)\right|+\frac{1}{|B'|}\left|\sum_{b \in B'_A}\mub_b(S)-q(S)\right|\nonumber\\
   & \leq 6\epsilon \sqrt{\frac{d\ln (e/ \epsilon)}{k}}+ \sqrt{2\epsilon}\sqrt{\hatvar_S(B')-\var_S(\hatq_{B'})+187\frac{\epsilon d\ln (e/ \epsilon)}{k}+\Big[q(S)-\hatq_{B'}(S)\Big]^2}\\
  & \leq 26\epsilon \sqrt{\frac{d\ln (e/ \epsilon)}{k}}+\sqrt{2\epsilon\left|\hatvar_S(B')-\var_S(\hatq_{B'})\right|}+\sqrt{2\epsilon}\left| \hatq_{B'}(S)-q(S)\right|.
\end{align*}
Noting that $2\epsilon \leq 1/8$ terminates the proof:
$$
\left| \hatq_{B'}(S)-q(S)\right|\leq 30\epsilon \sqrt{\frac{d\ln (e/ \epsilon)}{k}}+2\sqrt{\epsilon\left|\hatvar_S(B')-\var_S(\hatq_{B'})\right|}.
$$\hfill\( \Box \)

\subsection{Proof of Lemma \ref{lem:matrixexpression}, Matrix expression}\label{app:matrixexpression}

For each batch $b \in B$, define matrix $C_{b,B'}^{EV}$ as:
\begin{equation}\label{eq:defCb}
 \hatC_{b,B'}(j, l)=\Big[\mub_{b}(j)-\hatq_{B^{\prime}}(j)\Big]\Big[\mub_{b}(l)-\hatq_{B^{\prime}}(l)\Big], \ \forall (j,l) \in [d]^2 .   
\end{equation}

For each collection of batches $B'$ define
$$
\hatC(B') =\frac{1}{\left|B^{\prime}\right|}\sum_{b \in B'}\hatC_{b}.
$$
For a set $S\subseteq [d]$, define $\mathds{1}_S$ as the indicator vector of the elements in $S$. For any $S,S' \subseteq [d]$
\begin{align*}
    \left<\hatC(B'),\mathds{1}_S\mathds{1}_{S'}^T\right>&=\frac{1}{\left|B^{\prime}\right|}\sum_{b \in B'}\sum_{j\in S}\sum_{l \in S'}\Big[\mub_{b}(j)-\hatq_{B^{\prime}}(j)\Big]\Big[\mub_{b}(l)-\hatq_{B^{\prime}}(l)\Big]\\
    &=\frac{1}{\left|B^{\prime}\right|}\sum_{b \in B'}\left(\sum_{j\in S}\mub_{b}(j)-\sum_{j\in S}\hatq_{B^{\prime}}(j)\right)\left(\sum_{l\in S'}\mub_{b}(l)-\sum_{l\in S'}\hatq_{B^{\prime}}(l)\right)\\
    &=\widehat{\cov}_{S,S'}\left(B'\right).
\end{align*}

We can compute 
\begin{align*}
    \mathbb{E}\left[\sum_{j \in S}Z(j)\bigg | X\right]&= \lambda|S|\mathds{1}_{X \not \in S}+\left(\lambda(|S|-1)+1-\lambda\right)\mathds{1}_{X  \in S}\\
    &=\lambda|S|+\left(1-2\lambda\right)\mathds{1}_{X  \in S}.
\end{align*}
For a set $S$, let us define $Y_S= \big(\sum\limits_{j \in S}Z(j)\big)-q(S)$ and $\Delta_S=\lambda|S|-q(S)$. For any sets $S,S' \subseteq [d]$ s.t. $S\cap S'= \emptyset$, we have:
\begin{align*}
    \mathbb{E}\left[Y_SY_{S'}\right]&=\mathbb{E}\left[\mathbb{E}\left[Y_S\big | X\right]\mathbb{E}\left[Y_{S'}\big | X\right]\right]\\
    &=\mathbb{E}\left[\left(\Delta_S+\left(1-2\lambda\right)\mathds{1}_{X  \in S}\right)\left(\Delta_{S'}+\left(1-2\lambda\right)\mathds{1}_{X  \in S'}\right)\right]\\
    &=\Delta_{S'}\Delta_{S}+\Delta_{S}(1-2\lambda) p(S')+\Delta_{S'}(1-2\lambda) p(S) ~~  \text{ since $S\cap S'= \emptyset$}\\
    &=-\Delta_{S'}\Delta_S ~~~~ \text{ since by \eqref{eq:qS} we have $(1-2\lambda) p(S)=-\Delta_S$}.
\end{align*}

On the other hand, using the notation from Lemma \ref{lem:eqlaw}, we have:
\begin{align*}
    \mathbb{E}\left[Y_S^2\right] &= \mathbb{E}\bigg[\bigg(\sum_{j=1}^{|S|-1}(b_j - \E b_j)+b^S - \E b^S\bigg)^2\bigg] = \sum_{j=1}^{|S|-1}\V[b_j]+\V[b^S] \\
    &=(|S|-1)\lambda(1-\lambda)+\left(\lambda+(1-2\lambda)p(S)\right)\left(1-\lambda-(1-2\lambda)p(S)\right)\\
    &=(|S|-1)\lambda(1-\lambda)+\left(\lambda-\Delta_S\right)\left(1-\lambda+\Delta_S\right)\\
    &=-\Delta_S^2+|S|\lambda(1-\lambda)-(1-2\lambda)\Delta_S.
\end{align*}
For any $S,S' \subseteq [d]$, we thus have:
\begin{align*}
    \mathbb{E}\left[Y_SY_{S'}\right]&=    \mathbb{E}\left[\left(Y_{(S\cap S')}+Y_{(S\setminus S')}\right)\left(Y_{(S\cap S')}+Y_{(S'\setminus S)}\right)\right]\\
    &=\mathbb{E}\left[Y_{(S\cap S')}^2\right]+\mathbb{E}\left[Y_{(S\cap S')}Y_{(S\setminus S')}\right]+\mathbb{E}\left[Y_{(S\cap S')}Y_{(S'\setminus S)}\right]+\mathbb{E}\left[Y_{(S\setminus S')}Y_{(S'\setminus S)}\right]\\
    &=-\left(\Delta_{(S\cap S')}+\Delta_{(S\setminus S')}\right)\left(\Delta_{(S\cap S')}+\Delta_{(S'\setminus S)}\right)+|S\cap S'|\lambda(1-\lambda)-(1-2\lambda)\Delta_{(S\cap S')}\\
    &=-\Delta_{S}\Delta_{S'}+|S\cap S'|\lambda(1-\lambda)-(1-2\lambda)\Delta_{(S\cap S')}. \\
\end{align*}

For a vector $q$, define 
\begin{equation}\label{def_C}
    k\C(q)= -\left(\lambda\mathds{1}-q\right)\left(\lambda\mathds{1}-q\right)^T+\lambda(1-\lambda)I_d-(1-2\lambda)\text{Diag}(\lambda\mathds{1}-q).
\end{equation}

 For any two sets $S,S' \subseteq [d]$, we have: $
\mathbb{E}\left[Y_SY_{S'}\right]= \mathds{1}_S^Tk\C(q)\mathds{1}_{S'}$, so that $\mathbb{E}\left[\hatcov_{S,S'}(B'_G)\right]=\mathds{1}_S^T\C(q)\mathds{1}_{S'}$.
We now define: 
\begin{equation}\label{def_covSS}
    \cov_{S,S'}(B') := \mathds{1}_S^T\C(\hatq_{B'})\mathds{1}_{S'}.
\end{equation}

\subsection{Proof of Lemma \ref{lem:Groetendickapprox}, Grothendieck's inequality corollary }\label{app:Groetendickapprox}

\begin{proof}[Proof of Lemma \ref{lem:Groetendickapprox}]
\begin{itemize}
    \item For the first inequality, fix any $x,y \in \{0,1\}^d$ and three orthonormal vectors $e_0, e_1, e_2 \in \R^d$. Define the following vectors:
    \begin{align*}
        \forall j \in \{1,\dots, d\}: u^{(j)} = \begin{cases} e_0 ~~ \text{ if } x_j = 1,\\
        e_1 ~~ \text{ otherwise},
        \end{cases}
        ~~ \text{ and } ~~
         v^{(j)} = \begin{cases} e_0 ~~ \text{ if } y_j = 1,\\
        e_2 ~~ \text{ otherwise}.
        \end{cases}
    \end{align*}
    Then the matrix $M = \left[\langle u^{(i)}, v^{(j)}\rangle \right]_{ij}$ belongs to $\G$ and we have by construction $M = xy^T$ which proves the first inequality.
    \item For the second inequality, we have by Grothendieck's inequality
    \begin{align*}
        \max_{M \in \mathcal{G}} \langle M,A \rangle \leq 2 \max_{x,y \in \{\pm 1\}^d}\langle xy^T,A \rangle.
    \end{align*}
    For all $a \in \R$, define $a^+ = a \lor 0$ and $a^- = (-a) \lor 0$ and for all vector $x \in \R^d$, define $x^+ = (x_j^+)_j$ and $x^- = (x_j^-)_j$. Note that if $x \in \{\pm 1\}^d$, then $x^+, x^- \in \{0,1\}^d$. We therefore have:
    \begin{align*}
        \max_{x,y \in \{\pm 1\}^d}\langle xy^T,A \rangle &= \max_{x,y \in \{\pm 1\}^d} \big|\langle x^+y^{+T},A \rangle - \langle x^-y^{+T},A \rangle - \langle x^+y^{-T},A \rangle + \langle x^-y^{-T},A \rangle\big|\\
        & \leq 4 \max_{a,b\in \{0,1\}^d} \big|\langle ab^T,A \rangle\big|,
    \end{align*}
    which proves the second inequality.
\end{itemize}
\end{proof}

\subsection{Proof of Lemma \ref{lem:goodvsbad}, Score good vs. adversarial batches}\label{app:goodvsbad}

We first note that the Lemma implies the desired property for the batches in $B^o$, namely that each batch deletion has a probability at least $\frac{3}{4}$ of removing an adversarial batch. Indeed, we have:
$$
\sum_{b \in B^o} \varepsilon_b = \sum_{b \in B^o_G} \varepsilon_b +\sum_{b \in B^o_A} \varepsilon_b.
$$
If we had $\sum\limits_{b \in B^o_A} \varepsilon_b<7\sum\limits_{b \in B^o_G} \varepsilon_b$, this would imply:
$$
\sum_{b \in B^o} \varepsilon_b< 8\sum_{b \in B^o_G} \varepsilon_b<\sum_{b \in B'_A} \varepsilon_b,
$$
where the last inequality comes from the Lemma. However this is in contradiction with the definition of $B^o$, which is the sub-collection of $\epsilon|B|$ batches with top $\varepsilon_b$ scores, since $|B'_A| \leq \epsilon|B|$. We therefore have that $\sum\limits_{b \in B^o_A} \varepsilon_b\geq 7\sum\limits_{b \in B^o_G} \varepsilon_b$ hence $\sum\limits_{b \in B^o_G} \varepsilon_b \leq \frac{1}{8}\sum\limits_{b \in B^o} \varepsilon_b$. Denote by $B^o(t)$ the current set obtained from $B^o$ after having removed $t$ batches (and before $\sum\limits_{b \in B^o} \varepsilon_b$ has been halved). We keep deleting batches from $B^o$ until $\sum\limits_{b \in B^o(t)} \varepsilon_b \leq \frac{1}{2}\sum\limits_{b \in B^o} \varepsilon_b$. At each step, we therefore have that $\sum\limits_{b \in B^o_G} \varepsilon_b \leq \frac{1}{4}\sum\limits_{b \in B^o} \varepsilon_b$ hence the probability of removing a good batch from $B^o(t)$ is always less than $\frac{3}{4}$.

\paragraph{Subcase 1} We first prove the Lemma in the case where $\max_{S \subseteq [d]}\left|\hatq_{B'}(S)-\lambda |S|\right|\geq 11$. 

We have:
\begin{align*}
    |\hatq_{B'}(S)-\lambda |S||&\leq \frac{1}{|B'|}\; \Bigg| \sum_{b \in B'_G}\mub_b(S)-\lambda |S|\Bigg|+\frac{1}{|B'|} \Bigg| \sum_{b \in B'_A}\mub_b(S)-\lambda |S|\Bigg|\\
    &\leq \frac{|B'_G|}{|B'|} \frac{1}{|B'_G|} \;\Bigg| \sum_{b \in B'_G}\mub_b(S)-q(S)\Bigg|+ \frac{|B'_G|}{|B_G|}\left|\lambda |S|-q(S)\right|+\frac{1}{|B'|} \Bigg| \sum_{b \in B'_A}\mub_b(S)-\lambda |S|\Bigg|\\
     &\leq 6\epsilon\sqrt{\frac{d\ln(e/\epsilon)}{k}}+1+\frac{1}{|B'|} \Bigg| \sum_{b \in B'_A}\mub_b(S)-\lambda |S|\Bigg|  ~~ \text{ by equation \eqref{concent_first_moment}.}
\end{align*}
Let $S^*= \argmax_{S \subseteq [d]}|\hatq_{B'}(S)-\lambda |S||$. We have:
$$
\Bigg| \sum_{b \in B'_A}\mub_b(S^*)-\lambda |S^*|\Bigg| \geq 9(1-2\epsilon) |B_G|.
$$
On the other hand, by equation \ref{eq:meansmallcollection}, we have for any $B^{''}_G$ s.t. $|B''_G|\leq \epsilon |B_G|$:
\begin{align*}
    \Bigg| \sum_{b \in B^{''}_G}\mub_b(S^*)-\lambda |S^*|\Bigg|&\leq \Bigg| \sum_{b \in B^{''}_G}\mub_b(S^*)-q(S)\Bigg|+|B^{''}_G|\\
    &\leq (\epsilon+2\epsilon\sqrt{\frac{d\ln(e/\epsilon)}{k}})|B_G|\\
    &\leq (1+ \epsilon)|B_G|.
\end{align*}
Thus we have:
$$
\frac{\left| \sum_{b \in B'_A}\mub_b(S^*)-\lambda |S^*|\right|}{\left| \sum_{b \in B^{''}_G}\mub_b(S^*)-\lambda |S^*|\right|}\geq \frac{9(1-2\epsilon)}{1+\epsilon}> 8.
$$
\paragraph{Subcase 2} In the case where $\max_{S \subseteq [d]}\left|\hatq_{B'}(S)-\lambda |S|\right|\leq 11$, the proof relies on the following intermediary Lemma \ref{lem:intergoodvsbad}.
\begin{lemma}\label{lem:intergoodvsbad}
If conditions 1 and 2 hold, then, for any $B' \subset[B]$, for any to sets $S,S'$:
$$
\left(\tau_{B'}-11\sqrt{\tau_{B'}}-1313\right)\frac{\epsilon d \ln(e/\epsilon)}{k} \leq\frac{1}{|B'|}\sum_{b \in  B'_A}\langle M^*,\hatC_{b,B'} \rangle.
$$
\end{lemma}
\begin{proof}: In this proof only, we use the shorthand:
$$
\gamma:=\frac{\epsilon d \ln(e/\epsilon)}{k}.
$$
We have
$$
\langle M^*,D_{B'}\rangle= \langle M^*,\hatC(B')-\C(q)\rangle+\langle M^*,\C(q)-\C(\hatq_{B'} \rangle.
$$
We analyse separately each term.
For any $S',S$, according to Lemmas \ref{lem:C_lip} and equation \ref{eq:corruptionscoretoerror}, we have:
\begin{align*}
    \left|\text{Cov}_{S,S'}(\hatq_{B'})-\text{Cov}_{S,S'}(q)\right|&\leq \frac{11}{k}\max_{S^{''}}|\hatq_{B'}(S^{''})-q(S^{''})|\\
    &\leq\frac{(330+22\sqrt{\tau_{B'}})}{k} \epsilon \sqrt{\frac{d\ln (e/ \epsilon)}{k}}\\
    &\leq (330+22\sqrt{\tau_{B'}}) \gamma \sqrt{\frac{ 1}{d \ln(e/\epsilon)k}}\\
    &\leq (96+7\sqrt{\tau_{B'}}) \gamma .
\end{align*}
Where the last line come from $d \geq 3$, $\epsilon \leq \frac{1}{20}$. Thus, by Lemma \ref{lem:Groetendickapprox}, we have:
\begin{align}\label{eq:qtohatq}
     \argmax_{M \in \G} \langle M,\C(q)-\C(\hatq_{B'}) \rangle&\leq 8\max_{S,S'} \left|\text{Cov}_{S,S'}(\hatq_{B'})-\text{Cov}_{S,S'}(q)\right| \nonumber\\
     &\leq \frac{88}{k}\left(30+2\sqrt{\tau_B'}\right)\epsilon\sqrt{\frac{d\ln(e/\epsilon)}{k}}\nonumber\\
     &\leq (763+51\sqrt{\tau_{B'}})\gamma.
\end{align}

On the other hand,
\begin{align*}
    \hatC(B')-\C(q)&= \frac{1}{|B'|}\sum_{b\in B'}\hatC(b,B')-\C(q)\\
    &=\frac{1}{|B'|}\sum_{b\in B'_G}\hatC(b,B')-\C(q)+\frac{1}{|B'|}\sum_{b\in B'_A}\hatC(b,B')-\C(q).\\
\end{align*}
From Lemma \ref{lem:Groetendickapprox} and \ref{lem:covgoodbatches}, we have:
\begin{align*}
    \bigg|\Big\langle M^*, \frac{1}{|B_G'|}\sum_{b\in B'_G}\hatC(b,B')-\C(q)\Big\rangle\bigg|&\leq\bigg|\Big\langle M^*, \frac{1}{|B_G'|}\sum_{b\in B'_G}\hatC(b,B'_G)-\hatC(b,B')\Big\rangle\bigg|\\
    &+\bigg|\Big\langle M^*, \frac{1}{|B_G'|}\sum_{b\in B'_G}\hatC(b,B'_G)-\C(q)\Big\rangle\bigg|.
\end{align*}
We start by bounding the first term $A=\bigg|\Big\langle M^*, \frac{1}{|B_G'|}\sum_{b\in B'_G}\hatC(b,B'_G)-\hatC(b,B')\Big\rangle\bigg|$. By Lemma \ref{lem:Groetendickapprox}:
\begin{align*}
    A&\leq \frac{8}{|B_G'|}\max_{S,S' \in [d]} \bigg|\sum_{b \in B'_G}\left[\hatq_b(S)-\hatq_{B'}(S)\right]\left[\hatq_b(S')-\hatq_{B'}(S')\right]-\left[\hatq_b(S)-\hatq_{B'_G}(S)\right]\left[\hatq_b(S')-\hatq_{B'_G}(S')\right]\bigg|\\
    &= \frac{8}{|B_G'|}\max_{S,S' \in [d]} \bigg|\left[\hatq_{B'_G}(S)-\hatq_{B'}(S)\right]\left[\hatq_{B'_G}(S')-\hatq_{B'}(S')\right]\bigg|.
\end{align*}
By equation \ref{eq:contaminationtoerror} and condition 1, for any $S\subseteq[d]$, we have:
\begin{align*}
    \left|\hatq_{B'_G}(S)-\hatq_{B'}(S)\right| &\leq \left|\hatq_{B'_G}(S)-q(S)\right| +\left|q(S)-\hatq_{B'}(S)\right|\\
    &\leq \left(36 + 2 \sqrt{\tau_{B'}}\right)\epsilon\sqrt{\frac{d\ln(e/\epsilon)}{k}}.
\end{align*}
Thus,
\begin{align*}
    A&\leq 8\left(36 + 2 \sqrt{\tau_{B'}}\right)^2 \epsilon \gamma .
\end{align*}
By equation \ref{eq:boundcovgoodbatchesbis} and Lemma \ref{lem:Groetendickapprox}, we have:
$$
\bigg|\Big\langle M^*, \frac{1}{|B_G'|}\sum_{b\in B'_G}\hatC(b,B'_G)-\C(q)\Big\rangle\bigg| \leq 1408 \gamma.
$$
Thus:
\begin{align*}
    \bigg|\Big\langle M^*, \frac{1}{|B'|}\sum_{b\in B'_G}\hatC(b,B')-\C(q)\Big\rangle\bigg|&= \frac{B_G'}{|B'|}\bigg|\Big\langle M^*, \frac{1}{|B_G'|}\sum_{b\in B'_G}\hatC(b,B')-\C(q)\Big\rangle\bigg|\\
    &\leq 1408\gamma +8\left(36 + 2 \sqrt{\tau_{B'}}\right)^2 \epsilon \gamma .
\end{align*}

Finally, for any $q$, we have:
\begin{align*}
    \langle M^*, C(q)\rangle &\leq 8 \max_{S,S'} \cov_{S,S'}(q)\\
    &\leq \frac{8d}{k}.
\end{align*}
This gives:
\begin{align*}
    \left|\langle M^*,\hatC(B')-\C(q)\rangle\right| &\leq \left|\langle M^*,\frac{1}{|B'|}\sum_{b\in B'_G}\hatC(b,B')-\C(q)\rangle\right|+\left|\langle M^*,\frac{1}{|B'|}\sum_{b\in B'_A}\hatC(b,B')\rangle\right|\\
    & \quad +\frac{|B'_A|}{|B'|}\left|\langle M^*,\C(q)\rangle\right|\\
    &\leq 1408\gamma +8\left(36 + 2 \sqrt{\tau_{B'}}\right)^2 \epsilon \gamma +\frac{\epsilon}{1-2\epsilon}\frac{8d}{k}+\left|\langle M^*,\frac{1}{|B'|}\sum_{b\in B'_A}\hatC(b,B')\rangle\right|\\
    &\leq 1409\gamma +8\left(36 + 2 \sqrt{\tau_{B'}}\right)^2 \epsilon \gamma+\left|\langle M^*,\frac{1}{|B'|}\sum_{b\in B'_A}\hatC(b,B')\rangle\right|.
\end{align*}
We can now combine this with equations \ref{eq:qtohatq}:
\begin{align*}
    \tau_{B'}\gamma&= \langle M^*,D_{B'}\rangle\\
    &\leq \left|\langle M^*,\hatC(B')-\C(q)\rangle\right|+\left|\langle M^*,\C(q)-\C(\hatq_{B'} \rangle\right|\\
    &\leq 2200\gamma +8\left(36 + 2 \sqrt{\tau_{B'}}\right)^2 \epsilon \gamma+51 \sqrt{\tau_{B'}}\gamma+\frac{1}{|B'|}\bigg|\langle M^*,\sum_{b\in B'_A}\hatC(b,B') \rangle\bigg|.
\end{align*}
Thus:
$$
\Bigg|\langle M^*,\sum_{b\in B'_A}\hatC(b,B') \rangle\Bigg|\geq (1-2\epsilon)\left[(1-32\epsilon)\tau_{B'}-(\epsilon1152+51)\sqrt{\tau_{B'}}-2200-8*36^2\epsilon\right]|B_G|\gamma
$$

With $\epsilon \leq 1/100$, we get:
$$
\Bigg|\langle M^*,\sum_{b\in B'_A}\hatC(b,B') \rangle\Bigg|\geq (0.66 \tau_{B'}-62\sqrt{\tau_{B'}}-2260)|B_G|\gamma
$$

\end{proof}

On the other hand, for any collection of good batches $B^{''}_G\subseteq B'$ s.t. $|B^{''}_G|\leq \epsilon |B_G|$, we have by Lemma \ref{lem:Groetendickapprox}:
\begin{align*}
    \sum_{b \in  B^{''}_G}\langle M^*,\hatC_{b,B'} \rangle &\leq 8 \max_{S,S' \in [d]} \sum_{b \in B^{''}_G}\langle\indS{S} \indS{S'}^{T},\hatC_{b,B'} \rangle \\
    &= 8 \max_{S,S' \in [d]} \sum_{b \in B^{''}_G}\Big[\mub_b(S)-\hatq_{B'}(S)\Big]\Big[\mub_b(S')-\hatq_{B'}(S')\Big].
\end{align*}

We can decompose the terms in the sum:
\begin{align*}
    \Big[\mub_b(S)-\hatq_{B'}(S)\Big]\Big[\mub_b(S')-\hatq_{B'}(S')\Big]&= \Big[\mub_b(S)-q(S)\Big]\Big[\mub_b(S')-q(S')\Big]+ \Big[q(S)-\hatq_{B'}(S)\Big]\Big[q(S')-\hatq_{B'}(S')\Big]\\
    &+ \Big[\mub_b(S)-q(S)\Big]\Big[q(S')-\hatq_{B'}(S')\Big]+\Big[q(S)-\hatq_{B'}(S)\Big]\Big[\mub_b(S')-q(S')\Big].
\end{align*}
By condition 1:
$$
\max_{S,S' \in [d]} \sum_{b \in B^{''}_G}\Big[\mub_b(S)-q(S)\Big]\Big[\mub_b(S')-q(S')\Big] \leq 33 |B_G|\gamma.
$$
By equation \ref{eq:corruptionscoretoerror},
\begin{align*}
 \max_{S,S' \in [d]} \sum_{b \in B^{''}_G}  \Big[q(S)-\hatq_{B'}(S)\Big]\Big[q(S')-\hatq_{B'}(S')\Big] \leq |B^{''}_G|(33+2\sqrt{\tau_{B'}})^2\epsilon  \gamma.
\end{align*}
By equations \ref{eq:meansmallcollection} and \ref{eq:corruptionscoretoerror}, 
\begin{align*}
    \max_{S,S' \in [d]} \sum_{b \in B^{''}_G} \Big[q(S)-\hatq_{B'}(S)\Big]\Big[\mub_b(S')-q(S')\Big] &=\max_{S,S' \in [d]}| B^{''}_G|\left(\hatq_{B^{''}_G}(S')-q(S')\right)\Big[q(S)-\hatq_{B'}(S)\Big].\\
    &\leq 2(33+2\sqrt{\tau_{B'}})|B_G|\epsilon \gamma.
\end{align*}

Combining the three bounds we have:
\begin{align*}
    \sum_{b \in  B^{''}_G}\langle M^*,\hatC_{b,B'} \rangle &\leq  8|B^{''}_G|(33+2\sqrt{\tau_{B'}})^2\epsilon  \gamma+32(33+2\sqrt{\tau_{B'}})|B_G|\epsilon \gamma+264 |B_G|\gamma\\
    &\leq 8|B_G|(33+2\sqrt{\tau_{B'}})^2\epsilon^2  \gamma+32(33+2\sqrt{\tau_{B'}})|B_G|\epsilon \gamma+264 |B_G|\gamma\\
    &\leq \left[32\tau_{B'}\epsilon^2 +(1056\epsilon^2+64)\sqrt{\tau_{B'}}+(1056\epsilon+264) \right]|B_G|\gamma.
\end{align*}

Which gives with $\epsilon\leq 1/100$:
\begin{align*}
    \sum_{b \in  B^{''}_G}\langle M^*,\hatC_{b,B'} \rangle &\leq (0.0032\tau_{B'}+65 \sqrt{\tau_{B'}}+275) |B_G|\gamma.
\end{align*}
Thus, we have:
\begin{align*}
    \frac{\sum_{b \in  B^{'}_A}\langle M^*,\hatC_{b,B'} \rangle}{\sum_{b \in  B^{''}_G}\langle M^*,\hatC_{b,B'} \rangle}\geq \frac{0.66 \tau_{B'}-62\sqrt{\tau_{B'}}-2260}{0.02\tau_{B'}+65 \sqrt{\tau_{B'}}+275}.
\end{align*}

With $\sqrt{\tau_{B'}}\geq 200$,
\begin{align*}
    \frac{\sum_{b \in  B^{'}_A}\langle M^*,\hatC_{b,B'} \rangle}{\sum_{b \in  B^{''}_G}\langle M^*,\hatC_{b,B'} \rangle}\geq 8.
\end{align*}
\subsection{Auxiliary Lemmas}

\begin{lemma}[Covariance is Lipschitz]\label{lem:C_lip}
Let $q,q' \in \R^d$ and define $\epsilon = q' - q$. For any $S, S' \subset [d]$, ~ if $~\big|\epsilon(S)\big| \lor \big|\epsilon(S')\big| \leq 12$, then
$$ \left|\cov_{S,S'}\left(q\right)-\cov_{S,S'}\left(q'\right)\right| \leq \frac{15}{k} \max\left(\big|\epsilon(S)\big|, \big|\epsilon(S')\big|\right). $$
\end{lemma}

\begin{proof}[Proof of Lemma \ref{lem:C_lip}]
By equation~\eqref{eq:qS}, we have $\left|\Delta_S\right| \leq 1$ for all $S \subset [d]$. Therefore, by Lemma \ref{lem:matrixexpression} and equation \eqref{def_C}:
\begin{align*}
    \left|\cov_{S,S'}\left(q\right)-\cov_{S,S'}\left(q'\right)\right| &= \left|\left<\indS{S}\indS{S'}^T,\C(q)- \C(q')\right>\right| \\
    & = \frac{1}{k} \left|\left<\indS{S}\indS{S'}^T,qq^T - (q+\epsilon) (q+\epsilon)^T +\lambda \mathds{1}\epsilon^T +\lambda \epsilon\mathds{1}^T + (1-2\lambda)\text{Diag}(\epsilon) \right>\right|\\
    & = \frac{1}{k}\left|\epsilon(S)\Delta_{S'} + \epsilon(S')\Delta_{S} + (1-2\lambda)\epsilon(S\cap S')- \epsilon(S) \epsilon(S')\right|\\
    & \leq \frac{1}{k} \Big( \big|\epsilon(S)\big| + \big|\epsilon(S')\big| + \big|\epsilon(S)\big| + 12\big|\epsilon(S)\big|\Big)\\
    &\leq \frac{15}{k}\max \Big(|\epsilon(S)|, |\epsilon(S')|\Big).
\end{align*}
\end{proof}

If $22\epsilon\sqrt{\frac{d \ln(e/\epsilon}{k}}\geq 1$, the proven bound for the algorithm is trivially true.
Else, whenever condition 1 holds, we have for any $|B'_G|\geq (1-2 \epsilon)|B_G|$:
$$
\max |\hatq_{B'_G}(S)-q(S)| \leq 1.
$$
Thus, Lemma \ref{lem:C_lip} may be applied to $\cov_{S,S'}(\hatq_{B'_G})-\cov_{S,S'}(q)$.

%% file: contents/proof_cor_actual_estimator.tex
\section{Proof of Corollary \ref{cor:true_estim}}\label{proof_cor_actual_estim}

\begin{lemma}\label{lem:norm_equal_sup_sets}
Let $p \in \mathcal{P}_d$ and $p' \in \R^d$. Then $\sup\limits_{S \subseteq[d} |p(S)-p'(S)| \leq \|p - p'\|_1 \leq 2 \sup\limits_{S \subseteq[d} |p(S)-p'(S)|$.
\end{lemma}

\begin{proof}[Proof of Lemma \ref{proof_cor_actual_estim}]
The first inequality follows from the triangle inequality. For the second one, letting $A = \{j \in [d]: p_j \geq p'_j\}$, we have: $\|p-p'\|_1 = p(A) - p'(A) + p'(A^c) - p(A^c) \leq 2 \sup\limits_{S \subseteq[d} |p(S)-p'(S)|$.
\end{proof}

\begin{proof}[Proof of Corollary \ref{cor:true_estim}]
Let $\widehat{p}$ be the output of Algorithm \ref{alg:estimation_proc} and $\widehat{p}^* = \frac{\widehat{p}}{\|\widehat{p}\|_1}$. Then 
\begin{align*}
    \|p - \widehat{p}^*\|_1 &\leq \|\widehat{p} - p\|_1 + \|\widehat{p} - \widehat{p}^*|_1 =  \|\widehat{p} - p\|_1 + \left|\|\widehat{p}\|_1 - 1\right| \leq 2 \|p-\widehat{p}\|_1.
\end{align*}
\end{proof}

%% file: contents/LB_Robustness_Privacy.tex
\section{Lower bound: Proof of Proposition \ref{LB_Robustness_Privacy}}\label{sec:proof_LB}

For any two probability distributions $p,q$ over some measurable space $(\mathcal{X}, \mathcal{A})$, we denote by
$$ \chi^2(p||q) = \begin{cases} \int_{\mathcal{X}} \frac{p}{q} \, dp -1 \text{ if } p\ll q\\ +\infty \text{ otherwise } \end{cases} $$
the $\chi^2$ divergence between $p$ and $q$. We start with the following Lemma.

\begin{lemma}\label{lem:lower_bound}
Assume $d\geq 3$. There exists an absolute constant $c>0$ such that for all estimator $\hat{p}$ and all $\alpha$-LDP mechanism $Q$, there exists a probability vector $p \in \mathcal{P}_d$ satisfying
$$ \mathbb{E}\left[\sup_{z' \in \mathcal{C}(Z)} \big\| \hat{p}(z')-p\big\|_1\right]  \geq c\left\{\left(\frac{d}{\alpha\sqrt{k n}} + \frac{\epsilon \sqrt{d}}{\alpha \sqrt{k}}\right) \land 1\right\},$$
where the expectation is taken over all collections of $n'$ clean batches $Z^1, \dots, Z^{n'}$ where $Z^b = (Z_1^b,\dots, Z_k^b)$ and $Z_l^b \overset{iid}{\sim} Qp$.
\end{lemma}

This Lemma is the analog of Proposition \ref{LB_Robustness_Privacy} but with the guarantee in expectation rather than with high probability. We first prove this Lemma before moving to the proof of Proposition \ref{LB_Robustness_Privacy}.\\

\begin{proof}[Proof of Lemma \ref{lem:lower_bound}]
We first show that $R_{n,k}^*(\alpha, \epsilon, d) \geq c\left(\frac{d}{\alpha\sqrt{k n}} \land 1\right)$ for some  small enough absolute constant $c>0$. Informally, this amounts to saying that the estimation problem under both contamination and privacy is more difficult than just under privacy. Formally:
\begin{align*}
    R_{n,k}^*(\alpha, \epsilon, d) = \inf_{ \hat{p}, Q} \; \sup_{p \in \mathcal{P}_d} \; \mathbb{E}\left[\sup_{z' \in \mathcal{C}(Z))} \big\| \hat{p}(z')-p\big\|_1\right] \geq \inf_{ \hat{p}, Q} \; \sup_{p \in \mathcal{P}_d} \; \mathbb{E}\left[ \big\| \hat{p}-p\big\|_1\right] \geq c\left(\frac{d}{\alpha\sqrt{k n}} \land 1\right),
\end{align*}
where the last inequality follows from \cite{duchi2014local} Proposition 6. We also give a simpler proof of this fact in Appendix \ref{app:proofLBprivacyNoOutliers}, using Assouad's lemma.\\

We now prove $R_{n,k}^*(\alpha, \epsilon, d) \geq c\Big(\frac{\epsilon \sqrt{d}}{\alpha \sqrt{k}} \land 1\Big)$. 
For any $\alpha$-LDP mechanism $Q$ and probability vector $p \in \mathcal{P}_d$, denote by $Qp$ the density of the privatized random variable $Z$ defined by $Z | X \sim Q(\cdot|X)$ and by $Qp^{\otimes k}$ the density of the joint distribution of $k$ iid observations with distribution $Qp$. Define the set of pairs of probability vectors that are indistinguishable after privatization by $Q$ and adversarial contamination
\begin{equation}
    \Indist(Q) = \left\{(p,q) \in \mathcal{P}_d ~\big|~  TV(Qp^{\otimes k}, Qq^{\otimes k})\leq \epsilon\right\}.
\end{equation}
To derive the adversarial rate, it suffices to prove
\begin{equation}\label{program_adv_rate}
    \inf_{Q} \sup_{p,q \in \Indist(Q)} \|p-q\|_1 \geq c\left\{\frac{\epsilon \sqrt{d}}{\alpha \sqrt{k}} \land 1\right\}.
\end{equation}
To understand why \eqref{program_adv_rate} is a natural program to consider, fix an $\alpha$-LDP mechanism $Q$ and denote by $\left(\mathcal{Z}, \mathcal{U}, \nu\right)$ its image space. If $(p,q) \in \Indist(Q)$, then letting $$A = \frac{Qp^{\otimes k} \lor Qq^{\otimes k}}{1+ TV(Qp^{\otimes k}, Qq^{\otimes k})}, ~~~~ N^{(p)} = \frac{A- (1-\epsilon)Qp^{\otimes k}}{\epsilon}, ~~ \text{ and } N^{(q)} = \frac{A- (1-\epsilon)Qq^{\otimes k}}{\epsilon},$$
we can directly check that $A, N^{(p)}$ and $N^{(q)}$ are probability measures over $(\mathcal{Z}, \mathcal{U})$ (for $N^{(p)}$ and $N^{(q)}$, we use the fact that $(p,q) \in \Indist(Q)$ to prove that $N^{(p)}(dz) \geq 0$ and $N^{(q)}(dz) \geq 0$).
Moreover, it holds that $A = (1-\epsilon) Qp^{\otimes k} + \epsilon N^{(p)} = (1-\epsilon) Qq^{\otimes k} + \epsilon N^{(q)}$. 
This is exactly equivalent to saying that any clean family of $n$ batches with distribution $Qp^{\otimes k}$ or $Qp^{\otimes k}$ can be transformed into a $\epsilon$-contaminated family of $n$ batches with distribution $A$ through $\epsilon$ adversarial contamination. 
By observing such a contaminated family, it is therefore impossible to determine whether the underlying distribution is $p$ or $q$, so that the quantity $\|p-q\|_1/2$ is a lower bound on the minimax estimation risk. \\ 

We now prove \eqref{program_adv_rate}. For all $j \in \{1, \dots, d\}$ and $z \in \mathcal{Z}$, set 
\begin{align}
    q_j(z) &= \frac{Q(z|j)}{Q(z|1)} - 1, ~~~~  d\mu(z)= Q(z|1) d\nu(z), \label{def_q_mu}
\end{align}
and 
\begin{equation}\label{def_Omega}
    \Omega_Q = \left(\Omega_Q(j,j')\right)_{jj'} = \left(\int_{\mathcal{Z}} q_j(z) q_{j'}(z) d\mu(z)\right)_{ij}
\end{equation}
Given $Q$, we first prove that a sufficient condition for $(p,q)$ to belong to $\Indist(Q)$ is that $(p-q)^T \Omega (p-q) \leq C\epsilon^2/k$ for some small enough absolute constant $C>0$. Fix $p,q \in \mathcal{P}_d$ and define $\Delta = p-q$. By \cite{tsybakov2008introduction}, Section 2.4, we have
\begin{align}
    TV(Qp^{\otimes k}, Qq^{\otimes k}) \leq \sqrt{-1 + \left(1+ \chi^2(Qp || Qq)\right)^k}.\label{TV_chi2}
\end{align}
Now,
\begin{align*}
    \chi^2(Qp || Qq) &= {\displaystyle\int_{\mathcal{Z}}} \frac{\left(Qp(z) - Qq(z)\right)^2}{Qq(z)} dz = {\displaystyle\int_{\mathcal{Z}}} \frac{\left(\sum_{j=1}^d Q(z|j) \; \Delta_j\right)^2}{\sum_{j=1}^d Q(z|j) \; q_j} dz \\
    & = \int_{\mathcal{Z}} \frac{\left(\sum_{j=1}^d \Big(\frac{Q(z|j)}{ Q(z|1)} -1\Big) \; \Delta_j\right)^2}{ \sum_{j=1}^d \frac{Q(z|j)}{ Q(z|1)} \; q_j } \; Q(z|1)  d\nu(z) ~~~~ \text{ since } \sum_{j=1}^d \Delta_j = 0\\
    & \leq e^\alpha \int_{\mathcal{Z}} \sum_{j,j'=1}^d \Delta_j \Delta_{j'} q_j(z) q_{j'}(z) d\mu(z)\\
    & = e^\alpha \Delta^T \Omega_Q \Delta.
\end{align*}
Write $\Omega = \Omega_Q$ and assume that $\Delta^T \Omega_Q \Delta \leq C \epsilon^2/k$ for $C \leq e^{-2}$. Then equation \eqref{TV_chi2} yields:
\begin{align*}
    TV(Qp^{\otimes k}, Qq^{\otimes k}) &\leq \sqrt{-1 + \left(1+ e^\alpha \Delta^T \Omega \Delta\right)^k} \leq \sqrt{-1 + \exp\left(e^\alpha k \Delta^T \Omega \Delta\right)} \leq \sqrt{-1 + \exp\left(C e^\alpha \epsilon^2\right)} \leq \epsilon.
\end{align*}
Defining 
\begin{equation}
    \Indistkhi(Q) = \left\{(p,q) \in \mathcal{P} ~\Big|~ (p-q)^T \Omega (p-q) \leq \frac{C \epsilon^2}{k} \right\},
\end{equation}
it follows that $\Indistkhi(Q) \subset \Indist(Q)$ for all $Q$, so that
\begin{align*}
    \inf_{Q} \sup_{(p,q) \in \Indist(Q)} \|\Delta\|_1 \geq \inf_{Q} \sup_{(p,q) \in \Indistkhi(Q) } \|\Delta\|_1.
\end{align*}
Fix $Q$ and note that $\Omega_Q$ is symmetric and nonnegative. 
We sort its eigenvalues as $\{\lambda_1 \leq \dots \leq \lambda_d\}$ 
and denote by $v_1, \dots, v_d$ the associated eigenvectors. 
We also define $j_0 = \max\left\{j \in \{1,\dots, d\}: \lambda_j \leq 3e^2 \alpha^2\right\}$. 
Noting that $\forall j: |q_j| \leq e \alpha$ and that $\mu$ is a probability measure, we get that $Tr(\Omega) = \sum\limits_{j=1}^d \int_{\mathcal{Z}}q_j^2 d\mu \leq d e^2 \alpha^2$, 
so that $(d-j_0)3e^2\alpha^2 \leq d e^2 \alpha^2$ hence $j_0 \geq 2d/3$.\\

Let $H = \left\{x \in \R^d: x^T \mathbb
1 = 0\right\}$, and note that $V:=\text{span}(v_j)_{j\leq j_0} \cap H$ is of dimension at least $m=\frac{2d}{3}-1\geq \frac{d}{3}$. 
Therefore by Lemma \ref{norm_1_geq_norm_2}, there exists $\Delta \in V$ such that $\|\Delta\|_2^2 = \frac{C \epsilon^2}{2e^2\alpha^2k} \land \frac{1}{d}$ and $\|\Delta\|_1 \geq \CL{\ref{norm_1_geq_norm_2}} \sqrt{m}\|\Delta\|_2 \gtrsim \frac{\epsilon \sqrt{d}}{\alpha \sqrt{k}}$. 
Noting that over $\R^d$, $\|\cdot\|_1 \leq \sqrt{d}\|\cdot\|_2$, we also have $\|\Delta\|_1 \leq \frac{\epsilon}{\alpha \sqrt{k}}\land 1\leq 1$.\\

This allows us to define the following vectors: $p = \left(\frac{|\Delta_j|}{\|\Delta\|_1}\right)_{j=1}^d \in \mathcal{P}_d$ and $q = p-\Delta$. 
To check that $q \in \mathcal{P}_d$, note that the condition $\Delta^T \mathbb 1 = 0$ ensures that $q^T \mathbb 1 = 1$.
Moreover, for all $j \in \{1,\dots, d\}$ we have $q_j = \frac{|\Delta_j|}{\|\Delta\|_1} - \Delta_j \geq 0$ since $\|\Delta\|_1 \leq 1$.

Since by construction, we have $\Delta \Omega_Q \Delta \leq 2e^2 \alpha^2 \|\Delta\|_2^2 \leq \frac{C\epsilon^2}{k}$ and $p,q \in \mathcal{P}_d$, we have $(p,q) \in \Indistkhi(Q)$. For all $\alpha$-LDP mechanism $Q$, it therefore holds that $\sup\limits_{(p,q) \in \Indistkhi(Q)} \|\Delta\|_1 \gtrsim \frac{\epsilon \sqrt{d}}{\alpha \sqrt{k}} \land 1$. Taking the infimum over all $Q$, the result is proven.
\end{proof}

\begin{lemma}\label{norm_1_geq_norm_2}
There exists an absolute constant $\CL{\ref{norm_1_geq_norm_2}}$ such that for all $m \in \{\lceil \frac{d}{3}\rceil,\dots, d\}$ and all linear subspace $V \subset \R^d$ of dimension $m$, it holds:
$$ \sup_{v \in V} \frac{\|v\|_1}{\|v\|_2} \geq \CL{\ref{norm_1_geq_norm_2}} \sqrt{m}.$$ 
\end{lemma}

\begin{proof}[Proof of Lemma \ref{norm_1_geq_norm_2}]
Let $V$ be a linear subspace of $\R^d$ of dimension $m$ and denote by $\Pi_V := \left(\Pi_V(i,j)\right)_{ij}$ the orthogonal projector onto $V$. Let $X \sim \mathcal{N}(0, \Pi_V)$. For some large enough absolute constant $C>0$ we have:
\begin{align}
    \sup_{v \in V} \frac{\|v\|_1}{\|v\|_2} &\geq \mathbb E\left[\frac{\|X\|_1}{\|X\|_2}\right] \geq \mathbb E\left[\frac{\|X\|_1}{\|X\|_2} \mathbb 1 \left\{\|X\|_2 \leq C\sqrt{m}\right\}\right] \nonumber\\
    & \geq \underbrace{\frac{1}{C\sqrt{m}} \mathbb E \left[\|X\|_1\right]}_{\text{Principal term}} ~ - ~ \underbrace{ \frac{1}{C\sqrt{m}}\mathbb E \left[\|X\|_1 \mathbb 1 \left\{\|X\|_2>C \sqrt{m}\right\}\right]}_{\text{Residual term}} \label{decomp_princ_resid}
\end{align}
We first analyze the principal term.
\begin{align*}
    \mathbb E \|X\|_1 = \sum_{i,j=1}^d \mathbb E|X_{ij}| = \sqrt{\frac{2}{\pi}}\sum_{i,j=1}^d  \left|\Pi_V(i,j)\right|^{1/2}
\end{align*}
Note that $\forall i,j \in \{1, \dots, d\}: |\Pi_V(i,j)| \leq 1$ and that $\sum\limits_{i,j=1}^d \Pi_V^2(i,j) = m$. Therefore:
\begin{align}
    \inf_{dim(V) = m}\sum_{i,j=1}^d \left|\Pi_V(i,j)\right|^{1/2} &\geq ~ \inf_{A \in \R^{d \times d}} \sum_{i,j=1}^d |a_{ij}|^{1/2}  ~\text{ s.t. } \begin{cases}\|A\|_2^2 = m \\\forall i,j: |a_{ij}| \leq 1. \end{cases}\nonumber\\
    &= \inf_{a \in \R^{d \times d}} \sum_{i,j=1}^d a_{ij}  ~\text{ s.t. } \begin{cases}\sum_{i,j=1}^d a_{ij}^4 = m \\\forall i,j: 0 \leq a_{ij} \leq 1. \end{cases}
\end{align}
The last optimization problem amounts to minimizing an affine function over a convex set, hence the solution, denoted by $(a_{ij}^*)_{ij}$, is attained on the boundaries of the domain. Therefore, $\forall i,j \in \{1,\dots, d\}: a_{ij}^* \in \{0,1\}$. It follows from $\sum_{ij} a_{ij}^4 = m$ that the family $a_{ij}^*$ contains exactly $m$ nonzero coefficients, which are all equal to $1$. Therefore, the value of the last optimization problem is $m$, which yields that the principal term is lower bounded by $\frac{\sqrt{m}}{C}$.\\

We now move to the residual term. Writing $X = \sum_{j=1}^m x_j e_j$ where $(e_j)_{j=1}^m$ is an orthonormal basis of $V$, we have:
\begin{align}
    \mathbb E \left[\big\|X\big\|_1 \mathbb 1 \left\{\big\|X\big\|_2>C \sqrt{m}\right\}\right]& \leq \sqrt{d} \,  \mathbb E \left[\big\|X\big\|_2 \mathbb 1 \left\{\big\|X\big\|_2>C \sqrt{m}\right\}\right] \leq \sqrt{d} \,  \left\{\mathbb E \left[\big\|X\big\|_2^2 \mathbb 1 \left\{\big\|X\big\|_2^2>C^2 m\right\}\right]\right\}^{1/2} \nonumber\\
    & \leq \sqrt{d} \left\{m \,\mathbb E\left[ x_1^2 \mathbb 1 \Bigg\{\sum_{j=1}^m x_j^2 \geq C^2 m\Bigg\}\right]\right\}^{1/2}. \label{cool}
\end{align}
Moreover
\begin{align}
    \mathbb E\left[ x_1^2 \mathbb 1 \Bigg\{\sum_{j=1}^m x_j^2 \geq C^2 m\Bigg\}\right] &\leq \mathbb E \bigg[ x_1^2 \; \mathbb 1 \left\{x_1 \geq C\right\}\bigg] + \mathbb E\left[ x_1^2 \; \mathbb 1 \Bigg\{\sum_{j=2}^m x_j^2 \geq C^2 (m-1)\Bigg\}\right] \nonumber\\
    & \leq  \mathbb E \bigg[ x_1^2 \; \mathbb 1 \left\{x_1 \geq C\right\}\bigg] + \mathbb E \left[ x_1^2 \right]  \mathbb P \Bigg(\Big|\sum_{j=2}^m x_j^2 - \mathbb E x_1^2\Big| \geq \big(C^2 - \mathbb E x_1^2\big)(m-1)\Bigg)\label{cool2}
\end{align}
By the dominated convergence Theorem, $\lim\limits_{C\to +\infty} \mathbb E \big[ x_1^2 \; \mathbb 1 \left\{x_1 \geq C\right\}\big] = 0$. Moreover, by Chebyshev's inequality:
\begin{align}
    \mathbb P \Bigg(\Big|\sum_{j=2}^m x_j^2 - \mathbb E x_1^2\Big| \geq \big(C^2 - \mathbb E x_1^2\big)(m-1)\Bigg) \leq \frac{\mathbb V (x_1^2)}{\big(C^2 - \mathbb E x_1^2\big)^2(m-1)} \underset{C \to +\infty}{\to} 0.\label{cool3}
\end{align}
By \eqref{cool}, \eqref{cool2} and \eqref{cool3}, we conclude that for all absolute constant $c>0$, there exists a large enough absolute constant $C>0$ such that the residual term is at most $\frac{c\sqrt{d}}{C}$. Take $c=\frac{1}{2}$ and $m \geq \frac{d}{3}$, then by equation \eqref{decomp_princ_resid} we get:
\begin{align*}
    \sup_{v \in V} \frac{\|v\|_1}{\|v\|_2} &\geq \frac{\sqrt{m}}{C} - \frac{c\sqrt{d}}{C}\geq \Big(1 - \frac{\sqrt{3}}{2}\Big)\frac{\sqrt{m}}{C} =: \CL{\ref{norm_1_geq_norm_2}}\sqrt{m}.
\end{align*}
\end{proof}

\begin{proof}[Proof of Proposition \ref{LB_Robustness_Privacy}]

We distinguish between two cases.
\begin{enumerate}
    \item \textbf{\underline{First case}} If $\frac{d}{\alpha\sqrt{nk}} \leq \frac{\epsilon}{\alpha}\sqrt{\frac{d}{k}}$ i.e. if the dominating term comes from the contamination, taking $p,q \in \mathcal{P}_d$ like in the proof of Proposition \ref{lem:lower_bound} and $t \in \{p,q\}$ uniformly at random yields that
    \begin{align*}
        &\inf_{\hat p} \sup_{p \in \mathcal{P}_d}\mathbb{P}\left(\sup\limits_{Z \in \mathcal{C}(Y)}\|\hat{p}(Z) - p\|_1 \geq \|p-q\|_1/2\right) \\
        \geq & \inf_{\hat p} \mathbb{E}_{t \in \{p,q\}}\mathbb{P}_t\left(\sup\limits_{Z \in \mathcal{C}(Y)}\|\hat{p}(Z) - p\|_1 \geq \|p-q\|_1/2\right) \geq  \frac{1}{2} \geq O(e^{-d}),
    \end{align*}
    where $\|p-q\|_1 \gtrsim \frac{\epsilon}{\alpha} \sqrt{\frac{d}{k}} \land 1$.
    \item \textbf{\underline{Second case}} If $\frac{d}{\alpha\sqrt{nk}} \geq \frac{\epsilon}{\alpha}\sqrt{\frac{d}{k}}$ i.e. if the dominating term comes from the privacy constraint, then we set $N=nk$ and assume that we observe $Z_1,\dots, Z_N$ iid with probability distribution $Z | X \sim Q(\cdot|X)$ such that $X$ has a discrete distribution over $\{1,\dots,d\}$. In other words, the random variables $Z_i$ are no longer batches, but rather we have $nk$ iid clean samples that are privatized versions of iid samples with distribution $p$. By section \ref{app:proofLBprivacyNoOutliers}, it holds that 
    $$ \inf_{\hat{p}} \sup_{p \in \mathcal{P}_d}\mathbb{E}\left[\sup_{\text{contamination}}\|\hat{p} - p\|_1\right] \geq \inf_{\hat{p}} \sup_{p \in \mathcal{P}_d}\mathbb{E}\|\hat{p} - p\|_1 \geq c \frac{d}{\alpha \sqrt{N}},$$
for some small enough absolute constant $c>0$. We use the definition of $\gamma$ and of the cubic set of hypotheses $\mathcal{P}$ from \eqref{def_prior_privacy}.
Let $ \hat{p}$ be any estimator of the probability parameter and, for some small enough absolute constant $c>0$, define
\begin{equation}\label{def_r}
    r = c\, \frac{d}{\sqrt{kn}}.
\end{equation}
We first justify that for this particular set of hypotheses, it is possible to assume \textit{wlog} that
\begin{equation}\label{maximum_discrepancy}
    \|p - \hat{p}\|_1 \leq 6 \gamma d \,\leq\, 6 \, \cgamma r.
\end{equation} 
Indeed, define $u = \left(\frac{1}{d} \right)_{j=1}^d$. If for some observation $Z = (Z_1,\dots Z_N)$ the estimate $ \hat{p}(Z)$ satisfies $\| \hat{p}(Z) - u\|_1 > 4\gamma d$, then it is possible to improve $ \hat{p}$ by replacing it with the estimator $\bar{p}$ satisfying $\|\bar{p}(Z) - p\|_1 \leq 6 \gamma d$ and defined as:
\begin{align*}
    \bar{p} :=  \hat{p} \mathbb{1}\left\{\big\| \hat{p} - u\big\|_1 \leq 4 \gamma d\right\} + u \mathbb{1}\left\{\big\| \hat{p} - u\big\|_1 > 4 \gamma d\right\}.
\end{align*}
Indeed, recalling that $\forall p \in \mathcal{P}: \big\|u - p\big\|_1 = 2 \gamma d$, there are two cases.
\begin{itemize}
    \item If $\big\| \hat{p}(Z) - u\big\|_1 \leq 4\gamma d$, then $ \hat{p} = \bar{p}$ so that $\big\|p - \bar{p}\big\|_1 \leq \big\|p - u\big\|_1 + \big\|u - \bar{p}\big\|_1 \leq 2 \gamma d + 4 \gamma d = 6 \gamma d$.
    \item  Otherwise, $\bar{p} = u$ and we get
\begin{align*}
    \big\|\bar{p}(Z) - p\big\|_1 & = 2\gamma d = 4 \gamma d - 2 \gamma d< \big\| \hat{p}(Z) - u\big\|_1 -  \big\|u - p\big\|_1 \leq \big\| \hat{p}(Z) - p\big\|_1,
\end{align*}
\end{itemize}

which proves that \eqref{maximum_discrepancy} can be assumed \textit{wlog}. Now, from the proof of Lemma \ref{lem:LB_Privacy_no_outliers}, we also have
\begin{align}
    \sup_{p \in \mathcal{P}} \E_p\big\| \hat{p} - p\big\|_1 \, \geq \, \frac{\cgamma}{4} r =: Cr. \nonumber
\end{align}
Fix any $p \in \mathcal{P}$ and write $\pi := \sup\limits_{p \in \mathcal{P}} \mathbb{P}_p\left(\|\hat{p} - p\|_1 \geq cr\right)$ for $c = \cgamma \left(\frac{1}{4} - 6 \delta\right)> 0 $ for $\delta < \frac{1}{24} =: c'$.

It follows that:
\begin{align*}
    C r &\leq \sup_{p \in \mathcal{P}} ~ \E_p\big\| \hat{p} - p\big\|_1 \\
    & = \sup_{p \in \mathcal{P}}\left\{~ \E_p\Bigg[\big\| \hat{p} - p\big\|_1 \; \mathbb{1}\left\{\big\| \hat{p} - p\big\|_1 \geq cr\right\}\Bigg] ~+ ~ \E_p\Bigg[\big\| \hat{p} - p\big\|_1\; \mathbb{1}\left\{\big\| \hat{p} - p\big\|_1 < cr\right\}\Bigg]~\right\}\\
    & \leq 6 \cgamma r \cdot \pi \, + \, cr \hspace{3mm} \text{ by equation \eqref{maximum_discrepancy}, so that } ~ \pi \geq \frac{C-c}{6\cgamma} \geq \delta \geq O(e^{-d}).
\end{align*}

\end{enumerate}

\end{proof}

%% file: contents/LB_privacy_no_outliers.tex
\section{Simpler proof of the lower bound with privacy and no outliers}\label{app:proofLBprivacyNoOutliers}

Here, we assume that $k=1$ and that we observe $Z_1,\dots, Z_n$ that are $n$ iid with probability distribution $Z | X \sim Q(\cdot|X)$ and $X$ has a discrete distribution over $\{1,\dots,d\}$. We prove the following Lemma
\begin{lemma}\label{lem:LB_Privacy_no_outliers}
In this setting, it holds
\begin{equation}
    \inf_{\hat p} \sup_{p \in \mathcal{P}_d}\mathbb{E}\|\hat p - p\|_1 \geq c \frac{d}{\alpha \sqrt{n}},
\end{equation}
for some small enough absolute constant $c>0$.
\end{lemma}

For all $\epsilon \in \{\pm 1\}^{\lfloor d/2 \rfloor}$, define the probability vector $p_\epsilon \in \mathcal{P}_d$ such that 
\begin{equation}\label{def_prior_privacy}
    \forall j \in \{1,\dots, d\}: p_\epsilon(j) = \begin{cases}\frac{1}{d} + \epsilon_j \; \gamma & \text{ if } j \leq \frac{d}{2},\\
    \frac{1}{d} & \text{ if $d$ is odd and $j=\frac{d+1}{2}$,} \\
    \frac{1}{d} - \epsilon_{d-j+1} \gamma & \text{ otherwise, }
    \end{cases}
\end{equation}
where $\gamma = \frac{\cgamma}{\alpha \sqrt{n}} \land \frac{\cgamma}{d}$ and $\cgamma$ is a small enough absolute constant.
Consider the cubic set of hypotheses
\begin{equation}\label{def_set_prior_privacy}
    \mathcal{P} = \Big\{p_\epsilon ~\big|~ \epsilon \in \{\pm 1\}^{\lfloor d/2 \rfloor}\Big\}.
\end{equation}
This set $\mathcal{P}$ consists of $M = 2^{\lfloor d/2 \rfloor}$ hypotheses. 
Over $\mathcal{P}$, the $\ell_1$ distance simplifies as follows:
\begin{align}
    \forall \epsilon, \epsilon' \in \{\pm 1\}^{\lfloor d/2\rfloor}: \|p_\epsilon - p_{\epsilon'}\|_1 = 4 \gamma \; \rho(\epsilon, \epsilon'), \label{TV_Hamming}
\end{align}
where $\rho(\epsilon,\epsilon') = \sum\limits_{j=1}^{\lfloor d/2\rfloor }\mathbb{1}_{\epsilon_j \neq \epsilon'_j}$ denotes the Hamming distance between $\epsilon$ and $\epsilon'$. 

To apply Assouad's Lemma (see e.g. \cite{tsybakov2008introduction} Theorem 2.12.$(ii)$), let $\epsilon, \epsilon' \in \{\pm 1\}^{\lfloor d/2\rfloor}$ such that $\rho(\epsilon,\epsilon') = 1$. 
Recall that the observations $Z_1,\dots, Z_n$ are iid and follow the distribution $Z|X \sim Q(\cdot|X)$ where $Q$ is an $\alpha$-locally differentially private mechanism. 
Fix any such mechanism $Q$, and denote by $q_\epsilon$ and $q_{\epsilon'}$ the respective densities of $Z$ when $X \sim p_\epsilon$ and $X \sim p_{\epsilon'}$. 
We therefore have $\forall z \in \mathcal{Z}: q_\epsilon(z) = \int Q(z|x) \; p_\epsilon(x) d\nu(x)$ where $\nu$ denotes the counting measure over $\{1,\dots,d\}$. 
For some probability distribution $P$, we also denote by $P^{\otimes n}$ the law of the probability vector $(X_1, \dots, X_n)$ when $X_i \overset{iid}{\sim} P$. 
Now, we have: 
\begin{align}
    TV\left(q_\epsilon^{\otimes n}, q_{\epsilon'}^{\otimes n}\right) &\leq \sqrt{\chi^2\left(q_\epsilon^{\otimes n} \; || \; q_{\epsilon'}^{\otimes n}\right)}  = \sqrt{\left(1+\chi^2\left(q_\epsilon \; || \; q_{\epsilon'}\right)\right)^n -1} \label{TV_leq_chi2},
\end{align}
and defining $\Delta p(x) = p_\epsilon(x) - p_{\epsilon'}(x)$ for all $x \in \{1,\dots,d\}$, we can write:
\begin{align*}
    \chi^2\left(p_\epsilon \; || \; p_{\epsilon'}\right) &= \int_{\mathcal{Z}} \frac{\left(q_\epsilon(z) - q_{\epsilon'}(z)\right)^2}{q_{\epsilon'}(z)} dz = \int_{\mathcal{Z}} \frac{\left(\int Q(z|x) \; \Delta p(x)d\nu(x)\right)^2}{ \int Q(z|x) \; p_{\epsilon'}(x) d\nu(x)} dz \\
    & = \int_{\mathcal{Z}}  Q(z|1) \; \frac{\left({\displaystyle\int_{\mathcal{X}} } \Big(\frac{Q(z|x)}{ Q(z|1)} -1\Big) \; \Delta p(x)d\nu(x)\right)^2}{ {\displaystyle\int_{\mathcal{X}} } \frac{Q(z|x)}{ Q(z|1)} \; p_{\epsilon'}(x) d\nu(x)} dz ~~~~ \text{ since } \int \Delta p(x) d\nu(x) = 0\\
    & \leq \int_{\mathcal{Z}}  Q(z|1) \; \frac{\left({\displaystyle\int_{\mathcal{X}} }\Big|\frac{Q(z|x)}{ Q(z|1)} -1\Big| \; \big|\Delta p(x)\big| \, d\nu(x)\right)^2}{ {\displaystyle\int_{\mathcal{X}} } e^{-\alpha} \; p_{\epsilon'}(x) d\nu(x)} dz\\
    & \leq \int_{\mathcal{Z}}  Q(z|1) \; \frac{\left( C \alpha \;  TV(p_\epsilon, p_{\epsilon'})\right)^2}{ e^{-\alpha} } dz = e^\alpha C^2 \alpha^2 \left(2\gamma \;  \rho(p_\epsilon, p_{\epsilon'})\right)^2\\
    & \leq 12 C^2 \alpha^2 \gamma ^2 \leq  \frac{12 C^2 \cgamma^2}{n},
\end{align*}
where $C>0$ is an absolute constant such that for all $\alpha \in (0,1)$ we have $e^\alpha - 1 \leq C \alpha$ and $1 - e^{-\alpha} \leq C \alpha$. Now by \eqref{TV_leq_chi2}, we have:
\begin{align*}
    TV\left(q_\epsilon^{\otimes n}, q_{\epsilon'}^{\otimes n}\right) & \leq  \sqrt{\left(1+\chi^2\left(q_\epsilon \; || \; q_{\epsilon'}\right)\right)^n -1} \leq \sqrt{\exp\left(12 C^2 \cgamma^2\right)-1}. 
\end{align*}
Choosing $\cgamma$ small enough therefore ensures that $TV\left(q_\epsilon^{\otimes n}, q_{\epsilon'}^{\otimes n}\right) \leq \frac{1}{2}$, so that by Assouad's lemma, the minimax risk is lower bounded as:
\begin{align*}
    \inf_{\widehat p} \sup_{p \in \mathcal{P}_d} \E\|\widehat p - p \|_1 \geq \left\lfloor \frac{d}{2} \right\rfloor \frac{1}{2} 4 \gamma \left(1 - \frac{1}{2}\right) \geq \frac{\cgamma}{10} \left( \frac{d}{\alpha\sqrt{n}} \land 1\right).
\end{align*}